\documentclass[11pt,notitlepage,twoside]{article}
\usepackage{latexsym}
\usepackage{amsmath}
\usepackage{amsfonts}
\usepackage{xcolor}

\renewcommand{\a }{\alpha }
\renewcommand{\b }{\beta }

\newcommand{\D }{\Delta }

\renewcommand{\l }{\lambda }

\renewcommand{\O }{\Omega }

\newcommand{\be}{\begin{equation}}
\newcommand{\ee}{\end{equation}}
\newcommand{\ba}{\begin{align*}}
\newcommand{\ea}{\end{align*}}

\newenvironment{pf}{\noindent{\bf Proof.}\enspace}{%\rule{2mm}{2mm}
\hfill$\Box$\medskip}
\newenvironment{pfn}[1]{\noindent{\bf Proof of {#1}\enspace}}{%\rule{2mm}{2mm}
\hfill$\Box$\medskip}
\newcommand{\R}{\mathbb{R}}
\newtheorem{thm}{Theorem}[section]
\newtheorem{definition}{Definition}[section]

\newtheorem{lem}{Lemma}[section]

\newtheorem{prop}{Proposition}[section]
\newtheorem{rmk}{Remark}[section]

\newtheorem{corollary}{Corollary}[section]
\newtheorem{proof}{Proof}

\numberwithin{equation}{section}


\begin{document}
\title{
\vspace{0.5in} {\bf\Large The influence of a Generalized Hardy-Type
Potential on Finite Morse Index Solutions of the H\'enon Problem:
Local and Nonlocal Cases }}
%\title{
%\vspace{0.5in} {\bf\Large Nonexistence of Finite Morse Index
%Solutions for Local and Nonlocal H\'enon Problems with Generalized
%Hardy-Type Potentials}}
\author{
%%%%%%%%%%%%%%%%%%%%%%%%%%%%%%%%%%%%%%%
{\bf\large Habib Fourti}\vspace{1mm}\\
{\it\small Department of Mathematics and Statistics, College of
Science, }\\
{\it\small King Faisal University, Al-Ahsa 31982, Saudi Arabia}\\
{\it\small e-mail: hfourti@kfu.edu.sa  }\vspace{1mm}\\
%%%%%%%%%%%%%%%%%%%%%%%%%%%%%%%%%%%%%%%%%%%%%%%%%%
{\bf\large Abdelbaki Selmi}\vspace{1mm}\\
%{\it\small Northern Border university, Arar, }\\
%{\it\small Kingdom of Saudi Arabia}\\
{\it\small Universit\'{e} de Tunis, D\'{e}partement de Math\'{e}matiques},\\
{\it\small Facult\'{e} des Sciences de Bizerte},\\
{\it\small  Route soukra, Tunisie}\\
{\it\small e-mail:Abdelbaki.Selmi@fsb.rnu.tn  }\vspace{1mm}\\
{\bf\large  Cherif Zaidi}\vspace{1mm}\\
{\it\small LR21ES08: Analyse Non Lin\'{e}aire et G\'{e}om\'{e}trie, Facult\'{e} des Sciences de Tunis,}\\
{\it\small Universit\'{e} Tunis El Manar. \newline Campus Universitaire 2092, El Manar, Tunis, Tunisia},\\
{\it\small Institut Sup\'{e}rieur des Sciences Appliqu\'{e}es et de Technologie de Kairouan, }\\ {Avenue Beit El Hikma, 3100 Kairouan - Tunisia},\\
{\it\small e-mail:zaidicherif10@gmail.com  } \vspace{1mm}}

\date{\small{   } {\it\small     }}
\maketitle
\begin{abstract}
In this paper, we investigate nonexistence results for stable
solutions of the semilinear elliptic problem
$$(-\D)^s u+\lambda|x|^\alpha u=|x|^\beta |u|^{p-1}u \; \mbox{in}\; \mathcal{H}, $$
where $\mathcal{H}=\R^n$ or
 $\mathcal{H}=\R^n_+=\{x=(x',x_n),\, x'\in\R^{n-1},x_n>0\}$
, with Dirichlet conditions in the half-space case. We assume that
$n\geq2s$, $p>1$, $\lambda>0,$ $\alpha, \beta>-2s$ and
$0<s\leq1$.\\
 Our aim is to
analyze the influence of the Hardy-type power potential
$\lambda|x|^\alpha u$ in the presence of the H\'enon weight
$|x|^\beta$, with particular emphasis on the existence and
classification of weak stable solutions, including solutions that
may be unbounded or sign-changing.\\
We establish nonexistence results for nontrivial stable solutions
and, more generally, for solutions that are stable outside a compact
set, covering critical and supercritical regimes without additional
assumptions, and the subcritical regime under a suitable $L
^\infty$-smallness condition on a weighted version of the solution.
Our results apply to both the local case $s = 1$
 and the nonlocal case
$0 < s < 1 $, in the whole space and in the half-space.\\
The proofs are based on energy estimates combined with a Pohozaev
identity. In the  supercritical regime, we further establish a
monotonicity formula, following the approach of D\'avila et al.
(Trans. Amer. Math. Soc. 369, 6087-6104, 2017).

\end{abstract}
% we further establish a monotonicity formula that plays a
%crucial role in the analysis of the supercritical regime and in the
%classification of possible stable solutions.

\noindent
{\it \footnotesize 2000 Mathematics Subject Classification}. {\scriptsize  35J60, 35J65, 58E05}.\\
{\it \footnotesize Key words}. {Finite Morse index solution, H\'enon
problem, Hardy-type potential, Pohozaev identity, monotonicity
formula.}

\maketitle \thispagestyle{empty}

 \section{Introduction and main results}
In this paper, we investigate the nonexistence of nontrivial
solutions with finite Morse index of the following equation
\begin{equation}\label{n}
(-\D)^s u+\lambda|x|^\alpha u=|x|^\beta |u|^{p-1}u \; \mbox{in}\;
\mathcal{H},
\end{equation}
where $\mathcal{H}=\R^n$; or $\mathcal{H}=\R^n_+=\{x=(x',x_n),\,
x'\in\R^{n-1},x_n>0\}$ with, respectively, the complementary
Dirichlet condition $$u=0 \hbox{ in }\mathbb{R}^n \setminus \R^n_+$$
 for the nonlocal case and the Dirichlet boundary condition
$$u = 0\hbox{
 on }\partial \R^n_+ $$ for the local case. Here
$n\geq2s$, $p>1$, $\lambda>0,$ $\alpha, \beta>-2s$ and $0<s\leq1$.\\
Problem \eqref{n} is motivated by the combination of the spatially
weighted nonlinearity characteristic of the H\'enon equation
\cite{DancerDuGuo,FazlyWei,Henon} and the presence of an additional
potential term. The factor $| x |^\b$
 describes the spatial dependence of the nonlinear reaction, while
$\lambda|x|^\alpha u,$ with $\lambda
>
0$, introduces a generalized Hardy-type potential related to the
singular inverse-power potentials arising in Schr\"odinger equations
and mathematical physics
\cite{BrezisDupaigneTesei,HardyPoten,Dupaigne,Tesei,VZ}. The
fractional Laplacian further allows the model to incorporate
nonlocal diffusion
effects \cite{BucurValdinoci}, with $s = 1$ recovering the classical local setting.\\

In the last decades, problems related to the nonexistence of finite
Morse index sign-changing solutions for Lane-Emden equations on
unbounded domains of $\R^n$ have received a lot of attention (see
for example \cite{ChenCuiZhuoYuan,Damascelli,DavilaDupaigneWangWei,
DavilaDupaigneWei,Farina1,FazlyWei, HarrabiRahal, HarrabiRahal1})
and the references therein.
%The main results of this paper will be collected in Theorems \ref{th1} and \ref{th1.2} below and they will be concerned with Liouville type results for suitable solutions of \eqref{n}.

We recall two critical exponents which play an important role in the
sequel, namely the classical Sobolev exponent
\begin{eqnarray}\label{ps} p_S(n,2s, \b):=
\begin{cases}
+\infty\;\;&\text{if  $n\leq 2s$}\\
 \frac{n+2s+2\b}{n-2s}\;\; &\text{if $n> 2s,$}
\end{cases}
\end{eqnarray}
and the Joseph-Lundgren exponent
\begin{eqnarray} \label{pc}
p_c(n, \b)=\begin{cases} \infty \; &\text{ if \; $n\leq 10+4\b$} \\
\frac{(n-2)^2-2(\b+2)(\b+n)
+2\sqrt{(\b+2)^3(\b+2n-2)}}{(n-2)(n-4\b-10)}\;  &\text{if \; $n>
10+4\b$}.  \end{cases}
\end{eqnarray}
In what follows, we first present our main results in the local case
and then in the nonlocal case.
\subsection{The local case: $s=1$} For the case $s = 1,$ the nonlinear
elliptic equations with Hardy potential have been studied by many
authors (see \cite{ Dupaigne, JeongLee,Tesei} and the references
therein).

Our motivation for investigating \eqref{n} comes from the results of
\cite{Farina} in which Farina established Liouville-type theorems
for stable solutions and finite Morse index solutions of \eqref{n}
with $\lambda=\b=0$, which is called the Lane-Emden equation:
\begin{align}\label{e2}
  -\Delta u=|u|^{p-1}u \quad \mbox{in}\quad \R^n.
\end{align}
Recently, there has been some interest in studying stable or finite
Morse index solutions of the following autonomous elliptic equation:
\begin{align}\label{e3}
  -\Delta u=f(u) \quad \mbox{in}\quad \R^n.
\end{align}

Stable radial solutions of \eqref{e3} are well-understood by the
works in \cite{CabreCapella} and \cite{Villegas}. It was shown in
\cite{CabreCapella} that every bounded radial stable solution must
be constant if $n\geq10$ and $f\in C^1(\R^n)$ satisfies a generic
non-degeneracy condition. However, in case of nonradial solutions,
much less is known. In \cite{DupaigneFarina, DupaigneFarina1}, the
authors presented Liouville-type theorems for stable solutions and
finite Morse index solutions of \eqref{e3} with general convex and
non-decreasing nonlinearities $f\leq 0.$ In particular, a complete
analysis of stable solutions and finite Morse index solutions
(including solutions which are stable outside a compact set) is
provided for two important nonlinearities $f(u)=|u|^{p-1}u, p > 1$
and $f(u)=e^u$ in \cite{DancerFarina, Farina, Farina1}. When
$f(u)=|u|^{p-1}u$, Farina's results \cite{Farina} say that the
equation \eqref{e2} with $1<p<p_c(n, 0)$, has no nontrivial stable
solution and it admits a positive radial stable solution if $p\geq
p_c(n, 0)$ and $n\geq11.$ It was also proved that \eqref{e2} has no
nontrivial finite Morse index solution when $1<p<p_c(n, 0)$ and
$p\neq\frac{n+2}{n-2}$.

Furthermore, in \cite{DancerDuGuo}, Dancer, Du and Guo extended some
results of \cite{Farina}. They considered \eqref{n} with $s=1$ and
$\lambda=0$ as follows:
\begin{align}\label{e4}
  -\Delta u=|x|^\b|u|^{p-1}u \quad \mbox{in}\quad \O,
\end{align}
where $p>1,\; \b>-2$ and $\O$ is a bounded or unbounded domain of
$\R^n$ for $n\geq2.$ It was proved that there is no nontrivial
stable solution of \eqref{e4} in $\R^n$ if $1<p<p_c(n, \b)$ and that
for $p\geq p_c(n, \b),$ \eqref{e4} admits a positive radial stable
solution in $\R^n$. Moreover, for a finite Morse index solution $u,$
the authors obtained the behavior of $u$ near the origin when $\O$
is a punctured ball $B_R(0)\setminus \lbrace 0\rbrace$, as well as
its behavior at infinity when $\O$ is an exterior domain
$\R^n\setminus B_R(0)$. In addition, Wang and Ye \cite{WangYe}
obtained a Liouville-type result for finite Morse index solutions of
\eqref{e4} in $\R^n$, which is a partial extension of results in
\cite{DancerDuGuo}. On the other hand, Dancer obtained many results
related to stable or finite Morse index solutions of \eqref{n} (see
\cite{Dancer, Dancer1}). For other relevant papers, see
\cite{BahriLions, DavilaDupaigneFarina} and the references therein.

When $\a=-2$, Jeong and Lee \cite{JeongLee} expected that further
extension and generalization of the above results can be made. Hence
they considered a more general non-autonomous equation \eqref{n} in
an entire space $\R^n$ and a punctured space $\R^n\setminus \lbrace
0\rbrace$ to extend some results of \cite{DancerDuGuo} and
\cite{Farina} (partially including \cite{WangYe}). They verified
that there is an exponent $p=p_c(n,\b,\lambda)$ depending on $n$,
$\b$ and $\lambda$ such that \eqref{n} has no nontrivial stable
solution in $\R^n$ if $1<p<p_c(n,\b,\lambda)$ and it admits a
positive radial stable solution in $\R^n$ for certain range of the
exponent $p$ larger than $p_c(n,\b,\lambda)$. They also proved that
\eqref{n} has no nontrivial finite Morse index solution in $\R^n$
and $\R^n\setminus \lbrace 0\rbrace$ if $1<p<p_c(n,\b^-,\lambda)$
with $\b^-=\min(\b,0)$ and $p\neq\frac{n+2+2\b}{n-2}$.\\
%%%%%%%%%%%%%%%%%%%%%%%%%%%%%%%%%%%%%%%%%%%%%%%%%%%%%%%%%%%%%%%%%
The work of Jeong and Lee \cite{JeongLee} deals with the
Henon-Lane-Emden Equation involving the classical Hardy potential
($\lambda |x|^\a$ with $\a=-2$). In the present work, we extend this
framework by considering a generalized Hardy-type potential.\\
%%%%%%%%%%%%%%%%%%%%%%%%%%%%%%%%%%%%%%%%%%%%%%%%%%%%%%%%%%%%%%%%%%%%%%%%%
%%%%%%%%%%%%%%%%%%%%%%%%%%%%%%%%%%%%%%%%%%%%%%%%%%%%%%%%%%%%%%%%%%%%%%%%%%%%%%%%%%%%%%%%
 Before presenting our main result for the local case, let us
first provide the definition of stability.
\begin{definition}
We say that $u$ is a weak solution of \eqref{n} if $u\in
H^1_{loc}(\mathcal{H})\cap L^\infty_{loc}(\mathcal{H})$ and
\begin{eqnarray}
\int_{\mathcal{H}}\nabla u\nabla \varphi dx+\lambda|x|^{\a}u\varphi
-\int_{\mathcal{H}}|x|^{\b}|u|^{p-1}u\varphi=0,
\end{eqnarray}
for any $\varphi\in C^1_c(\mathcal{H}).$
\end{definition}
%By the standard elliptic regularity theory in [12], if $u$ is a weak solution of \eqref{n}, then
%$u\in H^2_{loc}(\R^n) \cap C^2(\R^n \setminus\{0\})$ and $u$ verifies \eqref{n} for all $x\in \R^n \setminus %%\{0\}$ in the classical sense. In this paper, we are mainly interested in stable and finite Morse index %solutions of \eqref{n} which are defined as follows (see [8, 9]).
\begin{definition}
Let $u$ be a weak solution of \eqref{n}. The associated quadratic
form of problem \eqref{n} is defined by
\begin{eqnarray}\label{stab}
Q_u (\varphi):=\int_{\mathcal{H}}|\nabla
\varphi|^2+\l\int_{\mathcal{H}}|x|^{\a}\varphi^2-p\int_{\mathcal{H}}|x|^\b|u|^{p-1}\varphi^2,\;\;\;\forall\varphi\in
C^1_c(\mathcal{H}).
\end{eqnarray}
We say that $u$ is

$\bullet$ stable if $Q_u(\varphi) \geq 0\,\,\,\forall\varphi\in
C^1_c(\mathcal{H})$,

$\bullet$ stable outside a compact set $\mathcal{K}$ if
$Q_u(\varphi)\geq0$ for any $\varphi\in
C^1_c(\mathcal{H}\backslash\mathcal{K})$,

$\bullet$ has a Morse index equal to $K\geq 1$ if $K$ is the maximal
dimension of a subspace $X_K$ of $C^1_c(\mathcal{H})$ such that
$Q_u(\varphi)<0$ for any $\varphi\in X_K\backslash\{0\}$.
\end{definition}
\begin{rmk}
(i) Clearly a solution is stable if and only if its Morse index is equal to zero.\\
(ii) Any finite Morse index solution $u$ is stable outside a compact
set $\mathcal{K}\subset\mathcal{H}$. Indeed there exist $K\geq1$ and
$X_K:=$ span$\{\varphi_1,\ldots,\varphi_K\}\subset
C^1_c(\mathcal{H})$ such that $Q_u(\varphi)<0$ for any $\varphi\in
X_K\backslash\{0\}$. Then, $Q_u(\varphi)\geq0$ for every $\varphi\in
C^1_c(\overline{\mathcal{H}}\backslash\mathcal{K})$, where
$\mathcal{K}:=\cup^K_{j=1}\mbox{supp}(\varphi_j)$.
\end{rmk}

The first main result of the present paper is the following.
\begin{thm}\label{th1.2}
Let $u $ be a weak solution to \eqref{n}. Assume that  $\lambda >0$
and $\a,\, \b>-2$, then  we have the following.
\begin{enumerate}
  \item Assume that $n\geq 2$. If $u$ is stable, then $u\equiv0$ for any $p>1$.
  \item Assume that $n\geq 3$. If $u$ is stable outside a compact set of $\mathcal{H}$,
  then we have
\begin{enumerate}
  \item If $1<p<p_S(n,2, \b)$ and $\big||x|^\frac{\b-\a}{p-1}u\big|_{L^\infty(\mathcal{H})}^{p-1}<\frac{\lambda (n+\a)(p+1)}{2(n+\b)}$, then  $u\equiv 0.$
  \item If $p \geq p_S(n,2, \b)$, then $u\equiv 0.$
\end{enumerate}
\end{enumerate}
\end{thm}
\begin{rmk}
\begin{enumerate}
  \item  Note that when $s=1$, assumption \eqref{cond1} is  \begin{eqnarray} \label{cond2}
p\frac{\Gamma(\frac{n}{2}-\frac{1+\frac{\b}{2}}{p-1})
\Gamma(1+\frac{1+\frac{\b}{2}}{p-1})}{\Gamma(\frac{1+\frac{\b}{2}}{p-1})\Gamma(\frac{n-2}{2}-\frac{1+\frac{\b}{2}}{p-1})}>\frac{\Gamma(\frac{n+2}{4})^2}{\Gamma(\frac{n-2}{4})^2}
\end{eqnarray} which is equivalent to $1<p<p_c(n,\b)$, where $
p_c(n,\b)$ is given by \eqref{pc}.

  \item  Our result shows that \eqref{cond2} is irrelevant when $\a,\b>-2$ and $\lambda>0$.
  \item For $\a=\b=0$, our $L^\infty$-smallness condition reduces to the one given
  in \cite{SelmiHarrabiZaidi}.
 \end{enumerate}
\end{rmk}
%%%%%%%%%%%%%%%%%%%%%%%%%%%%%%%%%%%%%%%%%%%%%%%%%%%%%%%%%%%%%%%%%%%%%
Our proof is based on the approach developed by Selmi, Harrabi and
Zaidi \cite{SelmiHarrabiZaidi}. More recently, they proved some
Liouville-type theorems for equation \eqref{n} with $\lambda>0$ and
$\a=\b=0$. Their results hold for classical solutions that are
stable outside a compact set in the case $n\geq3$ and
$p\geq\frac{n+2+2\b}{n-2}$, and for classical stable solutions in
the case $p>1$. Their proof is based on integral estimates involving
the Morse index, combined with the Pohozaev identity. Furthermore,
using some ideas from
\cite{DavilaDupaigneWangWei,DavilaDupaigneWei}, they established a
monotonicity formula to treat the supercritical case. In the present
work, we follow the same strategy and adapt their arguments
to the H\'enon-Lane-Emden equation with the generalized Hardy potential considered here.\\

%%%%%%%%%%%%%%%%%%%%%%%%%%%%%%%%%%%%%%%%%%%%%%%%%%%%%%%%%%%%%%%%%%%%%%%%%%%%%

%%%%%%%%%%%%%%%%%%%%%%%%%%%%%%%%%%%%%%%%%%%%%%%%%%%%%%%%%%%%%%%%%%%%%%%%%%%%%%%%%%%%%
%%%%%%%%%%%%%%%%%%%%%%%%%%%%%%%%%%%%%%%%%%%%%%%%%%%%%%%%
%%%%%%%%%%%%%%%%%%%%%%%%%%%%%%%%%%%%%%%%%%%%%%%%%%%%%%%%%%%%%%%%%%%%%%%%%%%%%%%%%%%%%%%%%%
%%%%%%%%%%%%%%%%%%%%%%%%%%%%%%%%%%%%%%%%%%%%%%%%%%%%%%%%%%%%%%%%%%%%%%%%%%%%%%%%%%%%%%%%%%%%%%%%%%%%%%%%%%%%%
\subsection{The nonlocal case: $0<s<1$}
Let $u\in L^2(\R^n),$ we will consider its Fourier transform
$$\hat{u}(\zeta)=\mathcal{F}(u)(\zeta)=\frac{1}{(2\pi)^{\frac{n}{2}}}\int_{\R^n} e^{-i\zeta\cdot x}u(x)dx.$$
For $0<s<1,$ the Sobolev space $H^s(\R^n)$ is defined as
$$H^s(\R^n)=\lbrace u\in L^2(\R^n),\; \;\; |\zeta|^s  \hat{u}\in L^2(\R^n)\rbrace.$$
So that the fractional Laplacian of $u$
$$(-\D)^s u:=\mathcal{A}_{n,s}\int_{\R^n}\frac{u(x)-u(y)}{|x-y|^{n+2s}}dy$$
is well-defined (in the principal-value sense) at every point
$x\in\R^n.$ The normalizing constant
$\mathcal{A}_{n,s}=\frac{2^{2s-1}}{\Pi^{n/2}}\frac{\Gamma(\frac{n+2s}{2})}{|\Gamma(-s)|}$
is of the order of $s(1-s)$ as $s$ converges to $0$ or $1.$ A great
deal of work is currently being devoted to the study of the
fractional Hardy-type equations, we refer the reader to the papers
\cite{AbdellaouiBentifour, AbdellaouiMedinaPeralPrimo,
DipierroMontoroPeralSciunzi, DipierroValdinoci, FrankLiebSeiringer}.
A more detailed mathematical introduction to the fractional
Laplacian is then presented in \cite{BucurValdinoci}.

For the case $0<s<1$ and $\lambda=0,$ Fazly and Wei \cite{FazlyWei}
gave a complete classification of finite Morse index solutions of
\eqref{n} via proving and applying a monotonicity formula. As a
matter of fact, they proved that for either $1 < p <
\frac{n+2s+2\b}{n-2s}$ or $p > \frac{n+2s+2\b}{n-2s}$ and
\begin{eqnarray}\label{cond1}
p\frac{\Gamma(\frac{n}{2}-\frac{s+\frac{\b}{2}}{p-1})
\Gamma(s+\frac{s+\frac{\b}{2}}{p-1})}{\Gamma(\frac{s+\frac{\b}{2}}{p-1})\Gamma(\frac{n-2s}{2}-\frac{s+\frac{\b}{2}}{p-1})}>\frac{\Gamma(\frac{n+2s}{4})^2}{\Gamma(\frac{n-2s}{4})^2},
\end{eqnarray}
the only finite Morse index solution is zero. This result is similar
to the one of \cite{DavilaDupaigneWei} concerning the the fractional
Lane-Emden equation (without the H\'enon weight function $|x|^\beta
$). On the other hand, the finite Morse index solutions to the
semi-linear elliptic equation \eqref{n} with $0<s<1$, $\lambda
>0$ and $\a=\b=0$ was examined in \cite{RahalZaidi}. When $0<s<1$,
$\b>\a=-2s$ and $\lambda>0$ recently, Zaidi \cite{Zaidi} studied
Equation \eqref{n} with the presence of the classical hardy
potential in the fractional setting, we state his result as follows:

\begin{thm}\label{thrz}\cite{Zaidi}
Let $u$ be a weak solution of \eqref{n} in $\mathbb{R}^n$ ($0<s<1$, $\b>\a=-2s$ and $\lambda>0$) which is stable outside a compact set. \\
$\bullet$  If $p>1$ and $p\neq p_S(n,2s, \b)$, then $u\equiv 0.$\\
$\bullet$  If $p= p_S(n,2s, \b)$, then $u$ has finite energy i.e.
$$\int_{\R^n}|(-\D )^{\frac{s}{2}}u|^2+ \lambda \int_{\R^n}|x|^{-2s} u^2=\int_{\R^n}|x|^{\b}|u|^{p+1}<+\infty.$$
If in addition $u$ is stable, then in fact $u\equiv 0$.
\end{thm}
We also recall that Kim and Lee \cite{KimLee} established a
classification result for finite Morse index solutions of the
fractional
H\'enon-Lane-Emden equation involving the classical Hardy potential $\gamma |x|^{-2s}$.
It is worth noting that all of the aforementioned results were obtained in the whole-space setting.\\
In this work, we are interested in knowing whether analogous
 classification results remain valid for weak solutions of problem
 \eqref{n},
involving the generalized Hardy potential $\l |x|^\a$ (with an
arbitrary $\a>-2s$), belonging to one of the following classes:
stable solutions, solutions which are stable
outside a compact set.\\% when $0<s\leq1$, $\lambda >0$ and
%$\a,\, \b>-2s$.\\

%%%%%%%%%%%%%%%%%%%%%%%%%%%%%%%%%%%%%%%%%%%%%%%%%%%%%%%%%%%%%%%%%%%%%%%%%%%%%%%%%%%%%%%%%%%%%%%%%%%%%%%
To describe our results more precisely, we first introduce the
following definitions.
\begin{definition}
We say that $u$ is a weak solution of \eqref{n} if $u\in
H^s_{loc}(\mathcal{H})\cap L^\infty_{loc}(\mathcal{H})$ and
\begin{eqnarray}
\frac{\mathcal{A}_{n,s}}{2}\int_{\mathcal{H}}\int_{\mathcal{H}}\frac{(u(x)-u(y))(\phi(x)-\phi(y))}{|x-y|^{n+2s}}dxdy+\lambda\int_{\mathcal{H}}|x|^{\a}u\phi
-\int_{\mathcal{H}}|x|^{\b}|u|^{p-1}u\phi=0,
\end{eqnarray}
for any $\phi\in C^1_c(\mathcal{H}).$
\end{definition}
\begin{definition}
We say that a solution $u$ of \eqref{n}\\
$\bullet$ is stable if
\begin{eqnarray} \label{h1}
Q_u
(\phi):=\frac{\mathcal{A}_{n,s}}{2}\int_{\mathcal{H}}\int_{\mathcal{H}}\frac{(\phi(x)-\phi(y))^2}{|x-y|^{n+2s}}dxdy+\lambda\int_{\mathcal{H}}|x|^{\a}\phi^2
-p\int_{\mathcal{H}}|x|^{\b}|u|^{p-1}\phi^2\geq 0,
\end{eqnarray}
for any $\phi\in C^1_c(\mathcal{H}).$\\
$\bullet$ has Morse index equal to $K\geq 1$ if $K$ is the maximal dimension of a subspace $X_K$ of $C_c^1( \mathcal{H})$ such that $Q_u(\phi)<0$ for any $\phi \in X_K \backslash \{0\}$.\\
$\bullet$ is stable outside a compact set $\mathcal{K}\subset
\mathcal{H}$, if $ Q_u (\phi)\geq0$ for any $\phi \in
C^1_c(\mathcal{H} \backslash \mathcal{K})$.
\end{definition}
\begin{rmk}
  \begin{enumerate}
    \item Clearly, a solution is stable if and only if its Morse index is equal to zero.
    \item It is well known that any finite Morse index solution $u$ is stable outside a compact set
    $\mathcal{K}\subset \mathcal{H}$. Indeed, there exist $K\geq1$ and $X_K:= \mbox{Span}\{\phi_1,..., \phi_K\}
    \subset C^1_c(\mathcal{H})$ such that $Q_u (\phi)<0$ for any $\phi \in X_K\backslash\{0\}$.
    Hence, $Q_u (\psi)\geq0$ for every $\psi \in C^1_c(\mathcal{\overline{H}} \backslash \mathcal{K}) $,
    where $\mathcal{K}:= \displaystyle\bigcup_{j=1}^K supp(\phi_{j})$.

  \end{enumerate}
\end{rmk}

Our second main result is the following.
\begin{thm}\label{th1}
Let $u$ be a weak solution of \eqref{n} where $0<s<1$, $n>2s$,
$\lambda
>0$ and $\a,\, \b>-2s$.
\begin{enumerate}
\item Assume that $u$ is stable outside a compact set of $\mathcal{H}$, then we have
\begin{itemize}\item  if $1<p<p_S(n,2s, \b)$ and $\big||x|^\frac{\b-\a}{p-1}u\big|_{L^\infty(\mathcal{H})}^{p-1}<\lambda\frac{\a+2s}{2}\Big(\frac{n+\b}{p+1}-\frac{n-2s}{2}\Big)^{-1}$,
then $u\equiv 0,$
\item  if $p\geq p_S(n,2s, \b)$, then $u\equiv 0.$
\end{itemize}
\item If $u$ is stable, then $u\equiv 0$ for all $p>1$.
\end{enumerate}
\end{thm}

\begin{rmk}
\begin{enumerate}
\item We may see clearly the influence of the linear term $\l|x|^\a u$ where $u$ is stable outside a compact set,
which requires the restrictive condition
$\big||x|^\frac{\b-\a}{p-1}u\big|_{L^\infty(\mathcal{H})}^{p-1}<\lambda\frac{\a+2s}{2}\Big(\frac{n+\b}{p+1}-\frac{n-2s}{2}\Big)^{-1}$
in the subcritical case i.e. $1<p<p_S(n,2s, \b),$ to use the
Pohozaev identity.
\item We point out that, as in the local case, such an $L^\infty$-smallness assumption is not required
for stable solutions.
\item Theorem \ref{th1} is sharp for $\a,\b> -2s$ and $\lambda>0.$
\end{enumerate}
\end{rmk}
%%%%%%%%%%%%%%%%%%%%%%%%%%%%%%%%%%%%%%%%%%%%%%%%%%%%%%%%%%%%%%%%%%%%%%%%%%%%%%%%%%%%%%
%%%%%%%%%%%%%%%%%%%%%%%%%%%%%%%%%%%%%%%%%%%%%%%%%%%%%%%%%%%%%%%%%%%%%%%%%%%%%%%%%%%%%%%%%%%%%%%%%%%%%%%%%%%%%%%%%%%%%%%%
To prove Theorem \ref{th1}, we make use of the $s$-harmonic
extension of the solution $u$, introduced by Caffarelli and
Silvestre \cite{CaffarelliSilvestre}. In fact, our approach relies
on the monotonicity formula and energy estimates satisfied by the
the extended solution.\\
%%%%%%%%%%%%%%%%%%%%%%%%%%%%%%%%%%%%%%%%%%%%%
This construction realizes the fractional Laplacian as a
Dirichlet-to-Neumann operator. More precisely, the $s$-harmonic
extension of $u$ to $\mathbb{R}^{n+1}_+$
 solves a degenerate elliptic equation, and its weighted normal derivative on
$\partial\mathbb{R}^{n+1}$
 recovers, up to a normalization constant, the fractional Laplacian of
$u$:
%%%%%%%%%%%%%%%%%%%%%%%%%%%%%%%%%%%%%%%%%%%%%%%%%%%%
\begin{thm}\cite{CaffarelliSilvestre, MolcanovOstrovskii, Spitzer}\label{0.1}
Take $s\in (0,1)$ and $u\in H^s(\R^n).$ For $X=(x,t)\in
\R^{n+1}_{+},$  let $$\widetilde{u}(X)=\int_{\R^n}P(X,y)u(y)dy,$$
where $$P(X,y)=p_{n,s}t^{2s}|(x-y, t)|^{-(n+2s)}$$  and $p_{n,s}$ is
chosen so that $\int_{\R^n}P(X,y)dy=1.$ Then, $\widetilde{u}\in
H^1(\R^{n+1}_+;t^{1-2s} dxdt)$ and
\begin{equation*}
\begin{cases} \nabla .(t^{1-2s}\nabla \widetilde{u})=0  &\text{in $\R^{n+1}_+$},\\
 \widetilde{u}=u  &\text{on $\partial\R^{n+1}_+$},\\
-\displaystyle \lim_{t\longrightarrow0}t^{1-2s}\partial_{t}
\widetilde{u}=\kappa_s(-\D)^s u  &\text{on $\partial\R^{n+1}_+$},
\end{cases}
\end{equation*}
where
\begin{eqnarray}
\kappa_s=\frac{\Gamma(1-s)}{2^{2s-1}\Gamma(s)}.
\end{eqnarray}
\end{thm}

Applying Theorem \ref{0.1} to a solution of \eqref{n} in $\R^n$, we
conclude that the extended function $\widetilde{u}(x, t),$ where $x
= (x_1,...,x_n)\in \R^n$ and $t\in \R_+,$ satisfies
\begin{equation}\label{00.3}
\begin{cases} -\nabla .(t^{1-2s}\nabla \widetilde{u})=0  &\text{in $\R^{n+1}_{+}$},\\
- \displaystyle{\lim_{{t\longrightarrow0}}} t^{1-2s}\partial_{t}
\widetilde{u}=-\kappa_s\lambda|x|^{\a}
\;\widetilde{u}+\kappa_s|x|^{\b}|\widetilde{u}|^{p-1} \widetilde{u}
&\text{on $\partial\R^{n+1}_{+}$}.
\end{cases}
\end{equation}
On the other hand, after some extension, problem \eqref{n}, when
$\mathcal{H}=\R^n_+$, is equivalent to the following local problem
in $\R^{n+1}_{++}$
\begin{equation}\label{0.3}
\begin{cases} -\nabla. (t^{1-2s}\nabla \widetilde{u})=0  &\text{in $\R^{n+1}_{++}$},\\
- \displaystyle{\lim_{{t\longrightarrow0}}} t^{1-2s}\partial_{t} \widetilde{u}=-\kappa_s\lambda|x|^{\a} \;\widetilde{u}+\kappa_s|x|^{\b}|\widetilde{u}|^{p-1} \widetilde{u} &\text{on $\R^{n}_{+}\times\{0\}$},\\
\widetilde{u}=0 &\text{on $\partial\R^{n}_{+}\times\R_+$}.
\end{cases}
\end{equation}
As in \cite[Lemma 2.4]{KimLee}, we have the following.
\begin{rmk}\label{extensionstab}Let u is a solution of \eqref{n}. Then if $u$
 is stable (respectively, stable outside a compact set), then so is its extension $\widetilde{u}$.
 \end{rmk}
%%%%%%%%%%%%%%%%%%%%%%%%%%%%%%%%%%%%%%%%%%%%%%%%%%%%%%%%%%%%%%%%%%%%%%%%%%%
%Due to the presence of potentials in our case, we prove various
%Liouville-type theorems for weak solutions under the assumption
%that they are stable or stable outside a compact set of $\R^n$.
Our approach to proving Theorem \ref{th1} is inspired by the work of
D\`avila et al. \cite{DavilaDupaigneWei} in the framework of the
fractional Lane-Emden equation. In fact, the proof relies on the
monotonicity formula and energy estimates satisfied by the extended
solution $\widetilde{u}$. The derivation of the required estimates
and the corresponding monotonicity formula will be the subject of
the next section.\\ We first establish standard integral estimates
as a consequence of the stability property. These estimates,
together with the Pohozaev identity, enable us to obtain the
nonexistence results in the subcritical and critical cases, namely,
for $1<p\leq p_S(n,2s,\b)$. We mention that a suitable $L
^\infty$-smallness condition on a weighted version of the solution is required in the subcritical regime.\\
We then consider the supercritical case $p>p_S(n,2s,\b)$. We derive
a monotonicity formula for equations \eqref{00.3} and \eqref{0.3}.
This formula plays a crucial role in the analysis of the
supercritical regime and in the classification of possible stable
solutions.\\

%%%%%%%%%%%%%%%%%%%%%%%%%%%%%%%%%%%%%%%%%%%%%%%%%%%%%%%%%%%%%%%%%%%%%%%%%%%%%%%%%%%%%%%
The rest of the paper is organized as follows. In Section 2, we
establish a monotonicity formula and derive the energy estimates
needed for the nonlocal case. Section 3 is devoted to the proof of
Theorem \ref{th1} for the fractional equation. In Section 4, we
establish a monotonicity formula and derive some estimates for the
local case. Finally, in Section 5, we prove Theorem \ref{th1.2} for
the local case.\\ Throughout the paper, we consider both the local
and nonlocal settings, with particular emphasis on the latter, which
presents additional difficulties and requires specific arguments.\\
The proofs of our results are carried out in the delicate setting of
Dirichlet problems in the half-space ($\mathcal{H}=\mathbb{R}^n_+$),
which requires a more careful analysis.
%%%%%%%%%%%%%%%%%%%%%%%%%%%%%%%%%%%%%%%%%%%%%%%%%%%%%%%%%%%%%%%%%%%%%%%%%%%%%%%%%%%%%%%%%%%%%%%%%%%%
\section{The monotonicity formula and Energy Estimates in the nonlocal case}
The monotonicity formula is a powerful tool to understand
supercritical elliptic equations or systems. This approach has been
used successfully for Lane-Emden equation in
\cite{DavilaDupaigneWangWei, DavilaDupaigneWei,HarrabiRahal,
SelmiHarrabiZaidi, RahalZaidi, Zaidi}.
\begin{thm}\label{th2}
Take a weak solution to \eqref{0.3} $\widetilde{u}\in
H^1_{loc}(\R^{n+1}_{++};t^{1-2s} dxdt)\cap
L^\infty_{loc}(\R^{n+1}_{++})$. For $\tau>0$, let
\begin{eqnarray}\label{fs1}
E(\widetilde{u};\tau)&:=&\tau^{\frac{2s(p+1)+2\b}{p-1}-n-1}\frac{2s+\b}{2(p-1)}\int_{\R^{n+1}_{++}\cap\partial
B^{n+1}_\tau}t^{1-2s}(\widetilde{u})^2
d\sigma\nonumber\\
&+&\tau^{\frac{2s(p+1)+2\b}{p-1}-n}\Big(\frac{1}{2}\int_{\R^{n+1}_{++}\cap
B^{n+1}_\tau}t^{1-2s}|\nabla \widetilde{u}|^2 dxdt+
\frac{\kappa_s\lambda}{2}\int_{\R^n_+\times\{0\}\cap  B^{n+1}_\tau}|x|^{\a}(\widetilde{u})^2 dx\nonumber\\
&&-\frac{\kappa_s}{p+1}\int_{\R^{n}_+\times\{0\}\cap
B_\tau^{n+1}}|x|^{\b}|\widetilde{u}|^{p+1}dx\Big)
 .
\end{eqnarray}
Then, $E$ is a nondecreasing function of $\tau$. Furthermore,
\begin{eqnarray}\label{fs2}
\frac{dE}{d\tau}&=&\tau^{\frac{2s(p+1)+2\b}{p-1}-n+1}\int_{\R^{n+1}_{++}\cap\partial B^{n+1}_\tau}t^{1-2s} \left(\frac{\partial \widetilde{u}}{\partial r}+\frac{2s+\b}{p-1}r^{-1}\widetilde{u}\right)^2 d\sigma\nonumber\\
&+&\tau^{\frac{2s(p+1)+2\b}{p-1}-n-1}\frac{(2s+\a)\kappa_s\lambda}{2}\int_{\R^n_+\times\{0\}\cap
B^{n+1}_\tau}|x|^\a(\widetilde{u})^2 dx.
\end{eqnarray}
\end{thm}
\begin{pf} The proof follows the main lines of the demonstration of
\cite[Theorem 1.4]{Zaidi}, with small modifications.
\end{pf}
\begin{rmk}

 In the above, $B^{n+1}_\tau $ denotes the Euclidean ball in $\R^{n+1}$ centered at $0$
of radius $\tau$, $\sigma$ the $n$-dimensional Hausdorff measure
restricted to $\partial B^{n+1}_\tau$, $r = |X|$ the euclidean norm
of a point $X = (x, t) \in \R^{n+1}_{++} ,$ and $\partial_r = \nabla
.\frac{X}{r}$ the corresponding radial derivative.

\end{rmk}

Before presenting some energy estimates, let us  briefly sketch the
proof of Theorem \ref{th1} in the supercritical regime and highlight
the role of the monotonicity formula: let $u$ be a solution of
\eqref{n} and let $\widetilde{u}$ denote its extension solving
\eqref{0.3}. The above monotonicity formula allow us to classify
solutions that are stable outside a compact set. To this end, we
first use the $L^{p+1}$ norm estimates established in Lemma
\ref{lem5}, and then show that the blow-down limit $
\displaystyle{\lim_{\tau \rightarrow \infty}}
\tau^{\frac{2s+\b}{p-1}} \widetilde{u}(\tau
X)=\widetilde{u}^{\infty}(X)=0$. Consequently, we deduce that
$\displaystyle\lim_{\tau\to+\infty} E(\widetilde{u},\tau) = 0$. On
the other hand, one easily verifies that $\displaystyle\lim_{\tau\to
0} E(\widetilde{u},\tau) \geq 0$. Since $E$ is nondecreasing, it
follows that $E(\widetilde{u},\tau)\equiv0$, and hence
$\frac{dE}{d\tau}\equiv0$. Using the boundary condition, we then
readily deduce that such solutions are trivial when $p>p_S(n,2s,\b)$
and $\lambda>0$.\\
 %For the critical exponent, we derive nonexistence
%result from the Pohozaev identity. For the case of $
%1<p<p_S(n,2s,\b)$, we apply the integral estimates and Pohozaev
%identity with the condition
%$\Big||x|^{\frac{\b-\a}{p-1}}\widetilde{u}(x,0)
%\Big|^{p-1}_{L^{\infty}(\mathcal{H})}<\lambda\frac{\a+2s}{2}\Big(\frac{n+\b}{p+1}-\frac{n-2s}{2}\Big)^{-1}$.

%\section{Energy estimates}

Now, we introduce some energy estimates for the solutions of
\eqref{n}, which are important when we perform a blow-down analysis
in the proof of Theorem \ref{th1}.

\begin{lem}\label{lem1}
Let $u$ be a weak solution to \eqref{n}. Assume that $u$ is stable
outside some ball $B_{R_{0}}\subset \R^n.$ Let $\eta\in
C^{\infty}_{c}(\R^n \setminus B_{R_{0}})$ and for $x \in \R^n,$
define
\begin{eqnarray}
\rho (x)=\int_{\R^n}\frac{(\eta(x)-\eta(y))^2}{|x-y|^{n+2s}}dy.
\end{eqnarray}
Then
\begin{align*}\frac{\mathcal{A}_{n,s}}{p}&\int_{\R^n}\int_{\R^n}\frac{(u(x)\eta(x)-u(y)\eta(y))^2}{|x-y|^{n+2s}}dxdy+\frac{\lambda}{p}\int_{\R^n}|x|^{\a}
u^2 \eta^2dx\\&+ \int_{\R^n}|x|^\b|u|^{p+1}\eta^2dx \leq
\frac{\mathcal{A}_{n,s}}{p-1}\int_{\R^n}u^2\rho dx.\end{align*}
\end{lem}
\begin{pf}
Multiplying the equation by $u\eta^2$ , integrating by parts and
using the stability of $u$, the desired result follows. For more
details we refer the reader to the proof of \cite[Lemma
2.1]{DavilaDupaigneWei}.
\end{pf}
We also cut the following results from \cite{DavilaDupaigneWei}.
\begin{lem}\label{lem2}\cite{DavilaDupaigneWei}
For $m > n/2$ and $x \in \R^n,$ let
\begin{eqnarray} \label{b2}
\eta(x)=(1+|x|^2)^{-m/2}
\;\;\;\mbox{and}\;\rho(x)=\int_{\R^n}\frac{(\eta(x)-\eta(y))^2}{|x-y|^{n+2s}}dy.
\end{eqnarray}
Then, there exists a constant $C = C(n, s,m) > 0$ such that
\begin{eqnarray}
C^{-1}(1+|x|^2)^{-\frac{n}{2}-s}\leq \rho(x)\leq
C(1+|x|^2)^{-\frac{n}{2}-s}.
\end{eqnarray}
\end{lem}
\begin{corollary}\label{cor}\cite{DavilaDupaigneWei}
Let $m > n/2$, $\eta$ given by \eqref{b2}, $R \geq R_0 \geq 1,$\;
$\psi \in C^\infty(\R^n)$ be such that $0 \leq\psi \leq 1,\;
\psi\equiv 0$ on $B_1$ and $\psi \equiv 1$ on $\R^n \setminus B_2.$
Let
\begin{eqnarray}
\eta_R(x)=\eta\left(\frac{x}{R}\right)\psi
\left(\frac{x}{R_0}\right)\;\;\;\mbox{and}\;\rho_R(x)=\int_{\R^n}\frac{(\eta_R(x)-\eta_R(y))^2}{|x-y|^{n+2s}}dy.
\end{eqnarray}
There exists a constant $C = C(n, s,m,R_0) > 0$ such that for all
$|x| \geq 3R_0$
$$\rho_R(x)\leq C \eta^2\left(\frac{x}{R}\right)|x|^{-(n+2s)}+R^{-2s}\rho\left(\frac{x}{R}\right).$$
\end{corollary}
%%%%%%%%%%%%%%%%%%%%%%%%%%%%%%%%%%%%%%%%%%%%%%%%%%%%%%%%%%%%%%%%%%%%%%%%%%%%%%%%%%%%%%%%%%%%%%%%%%%%%%%%%%%%%%%%%%%%%%%%%%%%%%%%%%%
%The proofs of Lemmas \ref{lem1}, \ref{lem2} and Corollary \ref{cor}, are similar to
%that of Lemmas 2.1, 2.2 and Corollary 2.3 in \cite{DavilaDupaigneWei} and we omit the details here.
%%%%%%%%%%%%%%%%%%%%%%%%%%%%%%%%%%%%%%%%%%%%%%%%%%%%%%%%%%%%%%%%%%%%%%%%%%%%%%%%%%%%%%%%%%%%%%%%%%%%%%%%%%%%%%%%%%%%%%%%%%%%%%%%%%%%%%%%%%
\begin{lem}\label{lem3}
Let $u$ be a weak solution of \eqref{n} which is stable outside a
ball $B_{R_0}$. Take $\rho_R$ as in Corollary \ref{cor} with $m\in
(\frac{n}{2},\;\frac{n}{2}+\frac{s(p+1)}{2}).$ Then, there exists a
constant $C=C(n, s,m, p,\b ,R_0)>0$ such that for all $R\geq 3R_0,$
$$\int_{\mathcal H}u^2\rho_R dx\leq C\left( \int_{\mathcal H\cap B_{3R_0}^n}u^2\rho_R dx+R^{n-\frac{2s(p+1)+2\b}{p-1}}\right).$$
\end{lem}
\begin{pf}By Corollary \ref{cor}, if $R \geq |x| \geq 3R_0,$ then
$$\rho_{R}(x)\leq C(|x|^{-n-2s}+R^{-2s})$$
and so
$$\int_{B_R \setminus B_{3R_0}^n}|x|^{-\frac{2\b}{p-1}}\rho_{R}(x)^{\frac{p+1}{p-1}}\eta_{R}(x)^{-\frac{4}{p-1}}dx \leq CR^{n-2s\frac{p+1}{p-1}-\frac{2\b}{p-1}} .$$
Similarly, if $|x| \geq R \geq 3R_0,$ then
$$\rho_R(x)\leq C R^{-2s}\left( 1 + \frac{|x|^2}{R^2}\right)^{-\frac{n}{2}-s}$$
and so
$$|x|^{-\frac{2\b}{p-1}}\rho_R(x)^{\frac{p+1}{p-1}}\eta_{R}(x)^{-\frac{4}{p-1}}
\leq C R^{-2s\frac{p+1}{p-1}-\frac{2\b}{p-1}}\left(
1+\frac{|x|^2}{R^2}\right)^{-\frac{n+2s}{2}\frac{p+1}{p-1}+\frac{2m}{p-1}}
.$$ Since $m\in (\frac{n}{2},\;\frac{n}{2}+\frac{s(p+1)}{2}),$ we
have $\frac{2m}{p-1}-\frac{n+2s}{2}\frac{p+1}{p-1}<-\frac{n}{2}$,
therefore \begin{equation}\label{H1}\int_{\mathcal H\setminus
B_{3R_0}^n}|x|^{-\frac{2\b}{p-1}}\rho_R(x)^{\frac{p+1}{p-1}}\eta_{R}(x)^{-\frac{4}{p-1}}dx
\leq CR^{n-2s\frac{p+1}{p-1}-\frac{2\b}{p-1}}. \end{equation} Now,
by splitting the integral, using H\"older inequality and tacking
into account of \eqref{H1}, we get
\begin{eqnarray}
&\displaystyle \int_{\mathcal H}&u^2\rho_Rdx\nonumber\\&=&\int_{
\mathcal H\cap B_{3R_0}^n}u^2\rho_R dx+\int_{\mathcal H\setminus
B_{3R_0}^n}|x|^{\frac{2\b}{p+1}}|x|^{-\frac{2\b}{p+1}} \;u^2 \;
\rho_R \;\eta_R^{-\frac{4}{p+1}} \eta^{\frac{4}{p+1}}_{R}
dx\nonumber\\&\leq &\int_{ \mathcal H\cap
B_{3R_0}^n}u^2\rho_Rdx+\left( \int_{\mathcal H\setminus
B_{3R_0}^n}|x|^\b|u|^{p+1}\eta^2_R dx \right)^{\frac{2}{p+1}}\left(
\int_{\mathcal H\setminus
B_{3R_0}^n}|x|^{-\frac{2\b}{p-1}}\rho^\frac{p+1}{p-1}_{R}\eta^{-\frac{4}{p+1}}_{R}dx
\right)^{\frac{p-1}{p+1}}\nonumber\\&\leq &\int_{ \mathcal H\cap
B_{3R_0}^n}u^2\rho_R(x)dx+CR^{(n-\frac{2s(p+1)+2\b}{p-1})\frac{p-1}{p+1}}\left(
\int_{\mathcal H}|x|^\b|u|^{p+1}\eta^2_R dx \right)^{\frac{2}{p+1}}.
\end{eqnarray}
By a standard approximation argument, Lemma \ref{lem1} remains valid
with $\eta =\eta_R$ and $\rho =\rho_R$ and so the result follows.
\end{pf}

\begin{lem}\label{lem4}
Assume that $p\neq\frac{n+2s+2\b}{n-2s}.$ Let $u$ be a weak solution
to \eqref{n} which is stable outside a ball $B_{R_0}$ and
$\widetilde{u}$ its extension, solving \eqref{0.3}. Then, there
exists a constant $C=C(n, p, s, \b,R_0, u)>0$ such that
$$\int_{B_R^{n+1}\cap \R^{n+1}_{++}}t^{1-2s}\widetilde{u}^2dxdt\leq CR^{n+2-\frac{2s(p+1)+2\b}{p-1}}$$
for any $R>3R_0.$
\end{lem}
The proof of Lemma \ref{lem4} is similar to that of \cite[Lemma
2.5]{DavilaDupaigneWei} and we omit the details.
\begin{lem}\label{lem5}
Let $u$ be a weak solution to \eqref{n} which is stable outside a
ball $B_{R_0}$ and $\widetilde{u}$ its extension, solving
\eqref{0.3}. Then, there exists a constant $C = C(n, p, s, \b, u) >
0$ such that
$$\int_{B_R^{n+1}\cap \R^{n+1}_{++}}t^{1-2s}|\nabla \widetilde{u}|^2dxdt+
\kappa_s \int_{B_R^{n+1}\cap (\R^{n}_+\times \{0\})}\lambda|x|^\a
\widetilde{u}^2\,+\,|x|^\b|\widetilde{u}|^{p+1}\,dx\leq C
R^{n-\frac{2s(p+1)+2\b}{p-1}},$$ for any $R>2R_0.$
\end{lem}

%%%%%%%%%%%%%%%%%%%%%%%%%%%%%%%%%%%%%%%%%%%%%%%%%%%%%%%%%%%%%%%%%
\begin{pf}
The $L^{2}$ and $L^{p+1}$ estimates of respectively
$|x|^\frac{\a}{2} \widetilde{u}$ and
$|x|^\frac{\b}{p+1}\widetilde{u}$ follow from Lemmas \ref{lem1} and
\ref{lem4}. Now take a cut-off function $\eta \in
C_{c}^{1}\left(\overline{\mathbb{R}_{++}^{n+1}}\right)$ such that
$\eta=1$ on $\mathbb{R}_{++}^{n+1} \cap\left(B_{R}^{n+1} \backslash
B_{2 R_{0}}^{n+1}\right)$ and $\eta=0$ on $B_{R_{0}}^{n+1} \cup
\mathbb{R}_{++}^{n+1} \backslash B_{2 R}^{n+1}$, and multiply Eq.
\eqref{0.3} by $\widetilde{u} \eta^{2}$. Then,
\begin{align}\label{2.6}
\kappa_{s} \int_{ \mathbb{R}_{+}^{n}\times \{0\}}|x|^\b
|\widetilde{u}|^{p+1} \eta^{2} d x & =\int_{\mathbb{R}_{++}^{n+1}}
t^{1-2 s} \nabla \widetilde{u}
 \nabla\left(\widetilde{u} \eta^{2}\right) d x d t+\kappa_{s} \lambda \int_{\mathbb{R}_{+}^{n}\times \{0\}}|x|^\a\widetilde{u}^{2} \eta^{2} d x
 \nonumber\\
& =\int_{\mathbb{R}_{++}^{n+1}} t^{1-2 s}\left(|\nabla(\widetilde{u}
\eta)|^{2}-\widetilde{u}^{2}|\nabla \eta|^{2}\right) d x d
t+\kappa_{s} \lambda \int_{\mathbb{R}_{+}^{n}\times \{0\}}|x|^\a
\widetilde{u}^{2} \eta^{2} d x.
\end{align}
Using Remark \ref{extensionstab}, since $u$ is stable outside
$B_{R_{0}}^n$, so is $\widetilde{u}$ and we deduce that
\begin{align*}\int_{\mathbb{R}_{++}^{n+1}} t^{1-2
s}\left(|\nabla(\widetilde{u} \eta)|^{2}-\widetilde{u}^{2}|\nabla
\eta|^{2}\right) d x d t\leq& \frac{1}{p}
\int_{\mathbb{R}_{++}^{n+1}} t^{1-2 s}|\nabla(\widetilde{u}
\eta)|^{2} d x d t\\&+\left(\frac{\kappa_{s} \lambda}{p}-\kappa_{s}
\lambda\right) \int_{\mathbb{R}_{+}^{n}\times \{0\}}|x|^\a
\widetilde{u}^{2} \eta^{2} d x.
\end{align*}
In other words,
\begin{equation}\label{2.7}
\int_{\mathbb{R}_{++}^{n+1}} t^{1-2 s}|\nabla(\widetilde{u}
\eta)|^{2} d x d t+\kappa_{S} \lambda \int_{\mathbb{R}_{+}^{n}\times
\{0\}}|x|^\a \widetilde{u}^{2} \eta^{2} d x \leq p^{\prime}
\int_{\mathbb{R}_{++}^{n+1}} t^{1-2 s} \widetilde{u}^{2}|\nabla
\eta|^{2} d x d t
\end{equation}
where $\frac{1}{p^{\prime}}+\frac{1}{p}=1$. We then apply Lemma
\ref{lem4}.
\end{pf}
%%%%%%%%%%%%%%%%%%%%%%%%%%%%%%%%%%%%%%%%%%%%%%%%%%%%%%%%
%%%%%%%%%%%%%%%%%%%%%%%%%%%%%%%%%%%%%%%%%%%%%%%%%%%%%%%%%%%%%%%%%%%%%%%%%
\section{Proof of Theorem \ref{th1}}
We begin by the following well known Pohozaev identity.
\begin{lem}
Let $\widetilde{u}$ be a weak solution of problem \eqref{0.3}, then
$\widetilde{u}$ satisfies the following identity
\begin{align}\label{p2}
&\frac{n-2s}{2}\int_{B_{R}^{++}}t^{1-2s}|\nabla \widetilde{u}|^2dxdt+\kappa_s\lambda\frac{n+\a}{2}\int_{ B_{R}^{+0}}|x|^{\a}(\widetilde{u}(x,0))^{2}dx-\kappa_s\frac{n+\b}{p+1}\int_{ B_{R}^{+0}}|x|^{\b}|\widetilde{u}(x,0)|^{p+1}dx\nonumber\\
&\quad\quad=\frac{1}{2}\int_{\partial B_{R}^{++}}|\nabla
\widetilde{u}|^2\langle t^{1-2s}(x,t),\nu\rangle dS-
\int_{\partial B_{R}^{++}}\langle \nabla \widetilde{u},(x,t)\rangle\langle t^{1-2s}\nabla \widetilde{u},\nu\rangle dS\nonumber\\
&\quad\quad\quad+\frac{\kappa_s\lambda}{2}\int_{ \partial
B_{R}^{+0}}|x|^\a(\widetilde{u}(x,0))^{2}\langle x,\nu\rangle dS-
\frac{\kappa_s}{p+1}\int_{ \partial
B_{R}^{+0}}|x|^{\b}|\widetilde{u}(x,0)|^{p+1}\langle x,\nu\rangle
dS,
\end{align}
where $R>0$ is an arbitrary constant,
\begin{align*}
B_R^{++}&=\{(x, t):|(x, t)| <R,\ x_n>0 ,\  t>0\},\\
 \partial
B_R^{++}&=\{(x, t):|(x, t)|=R,\ x_n>0 ,\   t>0\},\\
B_R^{+0}&=\{(x, t):|(x, t)|<R,\ x_n>0 ,\   t=0\}\hbox{ and }\\
\partial B_R^{+0}&=\{(x, t):|(x, t)|=R,\ x_n>0 ,\   t=0\}.
\end{align*}
\end{lem}
\begin{pf}
Multiplying the equation \eqref{0.3} by $\langle (x,t),\nabla
\widetilde{u}\rangle$ and integrating on $B_R^{++}$, then, we have
\begin{align}
0=&\int_{B_{R}^{++}}-\nabla(t^{1-2s}\nabla\widetilde{u} )\langle (x,t),\nabla \widetilde{u}\rangle dxdt\nonumber\\
=&\int_{B_{R}^{++}}t^{1-2s}\nabla \widetilde{u}\nabla \langle (x,t),\nabla \widetilde{u}\rangle dxdt-\int_{\partial B_{R}^{++}}\langle (x,t),\nabla \widetilde{u}\rangle\langle t^{1-2s}\nabla \widetilde{u},\nu\rangle dS\nonumber\\
&+\kappa_s\lambda\int_{B_{R}^{+0}}|x|^{\a}\;\widetilde{u}(x,0)\langle x,\nabla_x\widetilde{u}(x,0)\rangle dx-\kappa_s\int_{B_{R}^{+0}}|x|^{\b}|\widetilde{u}(x,0)|^{p-1}\widetilde{u}(x,0)\langle x,\nabla_x\widetilde{u}(x,0)\rangle dx\nonumber\\
=&\int_{B_{R}^{++}}t^{1-2s}(|\nabla\widetilde{u}|^2+\langle
(x,t),\frac{1}{2}\nabla(|\nabla \widetilde{u}|^2)\rangle) dxdt
-\int_{\partial B_{R}^{++}} \langle (x,t),\nabla \widetilde{u}\rangle\langle t^{1-2s}\nabla \widetilde{u},\nu\rangle dS\nonumber\\
&-\kappa_s\lambda\frac{n+\a}{2}\int_{
B_{R}^{+0}}|x|^{\a}(\widetilde{u}(x,0))^{2}dx
+\frac{\kappa_s\lambda}{2}\int_{ \partial
B_{R}^{+0}}|x|^\a(\widetilde{u}(x,0))^{2}\langle x,\nu\rangle dS
\nonumber\\&+\kappa_s\frac{n+\b}{p+1}\int_{
B_{R}^{+0}}|x|^{\b}|\widetilde{u}(x,0)|^{p+1}dx
-\frac{\kappa_s}{p+1}\int_{ \partial B_{R}^{+0}}|x|^{\b}|\widetilde{u}(x,0)|^{p+1}\langle x,\nu\rangle dS \nonumber\\
=&-\frac{n-2s}{2}\int_{B_{R}^{++}}t^{1-2s}|\nabla
\widetilde{u}|^2dxdt+ \int_{\partial B_{R}^{++}}\frac{1}{2}|\nabla
\widetilde{u}|^2\langle t^{1-2s}(x,t),\nu\rangle -
\langle \nabla \widetilde{u},(x,t)\rangle\langle t^{1-2s}\nabla \widetilde{u},\nu\rangle dS\nonumber\\
&-\kappa_s\lambda\frac{n+\a}{2}\int_{
B_{R}^{+0}}|x|^{\a}(\widetilde{u}(x,0))^{2}dx+
\frac{\kappa_s\lambda}{2}\int_{ \partial
B_{R}^{+0}}|x|^\a(\widetilde{u}(x,0))^{2}\langle x,\nu\rangle
dS\nonumber\\& +\kappa_s\frac{n+\b}{p+1}\int_{
B_{R}^{+0}}|x|^{\b}|\widetilde{u}(x,0)|^{p+1}dx
-\frac{\kappa_s}{p+1}\int_{ \partial
B_{R}^{+0}}|x|^{\b}|\widetilde{u}(x,0)|^{p+1}\langle x,\nu\rangle
dS,\nonumber
\end{align}
where we have used integration by parts formulas over $B_{R}^{++}$
and $B_{R}^{+0}$.
\end{pf}
%%%%%%%%%%%%%%%%%%%%%%%%%%%%%%%%%%%%%%%%%%%%%%%%%%%%%%%%%%%%%%%%%%%%%%%
%%%%%%%%%%%%%%%%%%%%%%%%%%%%%%%%%%%%%%%%%%%%%%%%%%%%%%%%%%%%%%%%%%%%%%%%
\subsection{Stable solutions outside a compact set in the subcritical and critical regimes}
We turn now to the Proof of Theorem \ref{th1} for stable solutions
outside a compact set in the subcritical and critical cases, namely $1\leq p \leq \frac{n+2s+2\b}{n-2s}$.\\
%\textbf{Proof of Theorem \ref{th1}}\
Let $\widetilde{u}$ be a solution of \eqref{0.3} which is stable
outside a compact. Since $\widetilde{u}(x,0)=u(x)$, then we can
apply Lemma \ref{lem5}, we deduce that \begin{equation}\label{g1}
\int_{\R^{n+1}_{++}}t^{1-2s}|\nabla \widetilde{u}(x,0)|^2dxdt=
\kappa_s \int_{\R^{n}_{+}}|(-\Delta )^{\frac{s}{2}}u|^2dx<+\infty,
\end{equation}
\begin{eqnarray}\label{g2'}
\int_{\R^n_+}|x|^\a(\widetilde{u}(x,0))^2dx =\int_{\R^n_+}|x|^\a
u^2dx <+\infty,
\end{eqnarray}

and
\begin{equation}\label{g2}
\int_{\R^n_+}|x|^\b|\widetilde{u}(x,0)|^{p+1}dx=
\int_{\R^{n}_{+}}|x|^\b|u(x)|^{p+1}dx<+\infty,
\end{equation}
then there exists a sequence $R_k\rightarrow\infty$ such that
$$\int_{\partial B_{R_k}^{++}}|\nabla \widetilde{u}|^2\langle t^{1-2s}(x,t),\nu\rangle dS\rightarrow0,$$
$$\int_{\partial B_{R_k}^{++}}\langle\nabla \widetilde{u},(x,t)\rangle\langle t^{1-2s}\nabla \widetilde{u},\nu\rangle dS\rightarrow0,$$
$$ \int_{\partial B_{R_k}^{+0}}|x|^\b|\widetilde{u}(x,0)|^{p+1}\langle x,\nu\rangle dS\rightarrow0,$$
and
$$\int_{ \partial B_{R_k}^{+0}}|x|^\a(\widetilde{u}(x,0))^{2}\langle x,\nu\rangle dS\rightarrow0,$$
as $k\rightarrow\infty$. So we deduce from the Pohozaev identity
that
\begin{align}\label{pz}
\frac{n-2s}{2}&\int_{\R^{n+1}_{++}}t^{1-2s}|\nabla\widetilde{u}(x,t)|^2dxdt+
\lambda\kappa_s\frac{n+\a}{2}\int_{\R^n_+}|x|^\a(\widetilde{u}(x,0))^2dx
\nonumber\\&=
\kappa_s\frac{n+\b}{p+1}\int_{\R^n_+}|x|^\b|\widetilde{u}(x,0)|^{p+1}dx.
\end{align}
However, using $\widetilde{u}$ as a test function for \eqref{0.3},
we get
\begin{equation}\label{ii}
\int_{\R^{n+1}_{++}}t^{1-2s}|\nabla
\widetilde{u}(x,t)|^2dxdt+\lambda\kappa_s\int_{\R^n_+}|x|^\a(\widetilde{u}(x,0))^2dx
=\kappa_s\int_{\R^n_+}|x|^\b|\widetilde{u}(x,0)|^{p+1}dx.
\end{equation}
Combining \eqref{pz} and \eqref{ii} gives
\begin{eqnarray} \label{iii}
\lambda\frac{\a+2s}{2}\int_{\R^n_+}|x|^\a(\widetilde{u}(x,0))^2
dx=\Big(\frac{n+\b}{p+1}-\frac{n-2s}{2}\Big)\int_{\R^n_+}|x|^\b|\widetilde{u}(x,0)|^{p+1}dx
.
\end{eqnarray}
Observe now since we assume that $p< \frac{n+2s+2\b}{n-2s}$ and
$\a,\b>-2s$ then we have
$\frac{\a+2s}{2}>0,\;\frac{n+\b}{p+1}-\frac{n-2s}{2}>0,$ and since
we have assumed that
$\Big||x|^{\frac{\b-\a}{p-1}}u(x)\Big|^{p-1}_{L^{\infty}(\R^n_+)}<\lambda\frac{\a+2s}{2}
\Big(\frac{n+\b}{p+1}-\frac{n-2s}{2}\Big)^{-1}$, there exists
$\varepsilon > 0 $ such that
$\Big||x|^{\frac{\b-\a}{p-1}}\widetilde{u}(x,0)\Big|^{p-1}_{L^{\infty}(\R^n_+)}=\Big||x|^{\frac{\b-\a}{p-1}}u(x)\Big|^{p-1}_{L^{\infty}(\R^n_+)}=\lambda(\frac{\a+2s}{2}-\varepsilon)
\Big(\frac{n+\b}{p+1}-\frac{n-2s}{2}\Big)^{-1}$ and by using
\eqref{iii}, we have
\begin{align*}\lambda\frac{\a+2s}{2}\int_{\R^n_+}|x|^\a(\widetilde{u}(x,0))^2dx&=
\Big(\frac{n+\b}{p+1}-\frac{n-2s}{2}\Big)
\int_{\R^n_+}|x|^\b|\widetilde{u}(x,0)|^{p+1}dx\\ &\leq \lambda
(\frac{\a+2s}{2}-\varepsilon)\int_{\R^n_+}|x|^\a(\widetilde{u}(x,0))^2dx.
\end{align*}
 So, we conclude that  $u\equiv0.$

 For $p = \frac{n+2s+2\b}{n-2s}$, since $\lambda\frac{n-2s}{2}>0$ and using \eqref{iii}, we get
 $$\lambda\frac{\a+2s}{2}\int_{\R^n_+}|x|^\a(\widetilde{u}(x,0))^2dx=0.$$ That means $u\equiv0.$\\

 Next, we focus on the study of stable solutions, since all the necessary ingredients have now been established and allow us to conclude the argument in the subcritical and critical cases.
 The study of stable solutions outside a compact set in the supercritical regime will be addressed subsequently.

 \subsection{Stable solutions}
%Before ending this section, % we are left with
We recall that our goal is to prove that there is no \textbf{stable}
nontrivial solutions for any $p>1$. We distinguish two cases:
$1<p\leq p_S(n,2s, \b)$ and $p> p_S(n,2s, \b)$. Since any stable
solution is also stable outside a compact set, the desired
nonexistence result in the supercritical regime will follow from the
analysis presented in the next subsection. In the sequel, we focus
on the first case, namely $1<p\leq p_S(n,2s, \b)$.
\\On one hand, we may apply the stability inequality \eqref{h1} with the test
function $\phi = u$, so that
\begin{equation}\label{iiii}p\int_{\R^n}|x|^{\b}|u|^{p+1} \leq
\frac{\mathcal{A}_{n,s}}{2}\int_{\R^n}\int_{\R^n}\frac{(u(x)-u(y))^2}{|x-y|^{n+2s}}dxdy+
\lambda\int_{\R^n}|x|^{\a}u^2. \end{equation} On the other hand,
using the following result given in \cite{CaffarelliSilvestre}:
$$\frac{\mathcal{A}_{n,s}}{2}\int_{\R^n}\int_{\R^n}\frac{(u(x)-u(y))^2}{|x-y|^{n+2s}}dxdy=
\frac{1}{\kappa_s}\int_{\R^{n+1}_{++}}t^{1-2s}|\nabla
\widetilde{u}(x,t)|^2dxdt$$ and since $p\leq p_S(n,2s, \b)$,
\eqref{ii}, through \eqref{g2'} and \eqref{g2}, ensures that
\begin{equation*}
\frac{\mathcal{A}_{n,s}}{2}\int_{\R^n}\int_{\R^n}\frac{(u(x)-u(y))^2}{|x-y|^{n+2s}}dxdy+\lambda\int_{\R^n_+}|x|^\a
u^2dx =\int_{\R^n_+}|x|^\b |u|^{p+1}dx.
\end{equation*}
This  contradicts  \eqref{iiii} unless $u\equiv0$.\\

 Before ending this section,  we are left with the classification of
 stable solutions outside a compact set in the supercritical regime.

\subsection{Stable solutions outside a compact set in the supercritical regime: Blow-down analysis}
\begin{pfn}{Theorem \ref{th1}}: the supercritical case.\\
Assume that $p >\frac{n+2s+2\b}{n-2s}.$ Take a solution $u$ of \eqref{n} which is stable outside the ball of radius $R_0$ and let $\widetilde{u}$ be its extension solving \eqref{0.3}. We split the proof into five steps :\\\\
\textbf{Step 1.} $\displaystyle{\lim_{\tau\longrightarrow +\infty}}E(\widetilde{u};\tau)<+\infty.$\\
Since $E$ is nondecreasing, it suffices to show that
$E(\widetilde{u};\tau)$ is bounded. Write $E = E_1 + E_2,$ where
$E_1$ is given by
\begin{eqnarray*}
E_1(\widetilde{u};\tau)&:=&\tau^{\frac{2s(p+1)+2\b}{p-1}-n}\Big(\frac{1}{2}\int_{\R^{n+1}_{++}\cap
B_\tau^{n+1}}t^{1-2s}|\nabla \widetilde{u}|^2 dxdt
+\frac{\kappa_s\lambda}{2}\int_{\R^{n}_+\times\{0\}\cap
B_\tau^{n+1}}|x|^{\a}(\widetilde{u})^2 dx
\nonumber\\&&-\frac{\kappa_s}{p+1} \int_{\R^{n}_+\times\{0\}\cap
B_\tau^{n+1}}|x|^{\b}|\widetilde{u}|^{p+1}dx\Big)
\end{eqnarray*}
 and
$$E_2(\widetilde{u};\tau):=\tau^{\frac{2s(p+1)+2\b}{p-1}-n-1}\frac{2s+\b}{2(p-1)}\int_{\R^{n+1}_{++}\cap\partial B_\tau^{n+1}}t^{1-2s}(\widetilde{u})^2  d\sigma .$$
By Lemma \ref{lem5}, $E_1$ is bounded. Since $E$ is nondecreasing,
$$E(\widetilde{u};\tau)\leq \frac{1}{\tau}\int^{2\tau}_\tau E(u;t)dt\leq C+\tau^{\frac{2s(p+1)+2\b}{p-1}-n-1}\int_{\R^{n+1}_+\cap B^+_{2\tau}}t^{1-2s}(\widetilde{u})^2  . $$
Applying Lemma \ref{lem4}, we deduce that $E$ is bounded.\\\\
\textbf{Step 2.}For $\tau>0$, define the function $u^\tau$ by
$$u^\tau(x)=\tau^\frac{2s+\b}{p-1}u(\tau x)\;\;\mbox{for}\;\; x\in \mathcal{H}.$$
There exists a sequence $\tau_i\longrightarrow +\infty$ such that
$(\widetilde{u^{\tau_i}})$ converges weakly in $
H^1_{loc}(\R^{n+1}_{++};t^{1-2s} dxdt)$ to a function
$\widetilde{u}^\infty.$ This follows from the fact that
$(\widetilde{u^{\tau_i}})$ is bounded in $
H^1_{loc}(\R^{n+1}_{++};t^{1-2s} dxdt)$ by Lemmas \ref{lem4} and
\ref{lem5}.  \\\\
\textbf{Step 3.}
$\widetilde{u}^\infty\equiv 0$.\\
On one hand, $\widetilde{u^\tau}$ satisfies the following PDE
\begin{equation}
\begin{cases} -\nabla .(t^{1-2s}\nabla \widetilde{u^\tau})=0  &\text{in
$\R^{n+1}_{++}$,}\\
\widetilde{u^\tau}=(\kappa_s\lambda)^{-1}\tau^{-(2s+\a)}|x|^{-\a}\left(\displaystyle{\lim_{{t\longrightarrow0}}}
t^{1-2s}\partial_t
\widetilde{u^\tau}+\kappa_s|x|^{\b}|\widetilde{u^\tau}|^{p-1}\widetilde{u^\tau}
\right) &\text{on $\R^{n}_+\times\{0\}$,}\\
\widetilde{u^\tau}=0&\text{on $\partial\R^{n}_+\times\R^{n}_+$.}
\end{cases}
\end{equation}
Taking limits in the sense of distributions, it follows that
$\widetilde{u}^\infty $ is a solution of the following problem:
\begin{equation}\label{limitpb}
\left\{\begin{array}{lll} -\nabla .(t^{1-2s}\nabla
\widetilde{u}^\infty)=0  \;\;\;\;\; \;\;\;\;\;\mbox{in}\
\R^{n+1}_{++},\\
\widetilde{u}^\infty=0 \;\;\;\;\; \;\;\;\;\; \;\;\;\;\;\;\;\;\;\;
\;\;\;\;\; \mbox{on}\
\partial\R^{n+1}_{++}.
\end{array}
\right.
\end{equation}
On the other hand, we claim that $\widetilde{u}^\infty $ is
homogenous. In fact, using  Step 1, the scaling property of $E$ i.e
$$ E(\widetilde{u} ; \tau R)=E\left(\widetilde{u^{\tau}} ; R\right)
\hbox{ for any }\tau>0 \hbox{ and } R>0,$$ and the monotonicity of
$E$ from Theorem \ref{th2}, we obtain that for any $R_{2}>R_{1}>0$,
$$
\begin{aligned}
0 & =\lim _{i \rightarrow+\infty}\left\{E\left(\widetilde{u} ;
\tau_{i} R_{2}\right)-E\left(\widetilde{u} ; \tau_{i}
R_{1}\right)\right\}=\lim _{i
\rightarrow+\infty}\left\{E\left(\widetilde{u^{\tau_{i}}} ;
R_{2}\right)-E\left(\widetilde{u^{\tau_{i}}} ; R_{1}\right)\right\} \\
& \geq \liminf _{i \rightarrow+\infty} \int_{B_{R_{2}}^{++}
\backslash B_{R_{1}}^{++}} t^{1-2 s} r^{2 s \frac{p+1}{p-1}+\frac{2
\beta}{p-1}-n+1}\left(r \partial_{r}
\widetilde{u^{\tau_{i}}}+\frac{2 s+\beta}{p-1}
\widetilde{u^{\tau_{i}}}\right)^{2} d x d t
\end{aligned}
$$
Thus, by the convergence of
$\left\{\widetilde{u^{\tau_{i}}}\right\}$ as $i \rightarrow \infty$,
established in Step 2, we obtain that, for any $R_{2}>R_{1}>0$,
$$
\int_{B_{R_{2}}^{++} \backslash B_{R_{1}}^{++}} t^{1-2 s} r^{2 s
\frac{p+1}{p-1}+\frac{2 \beta}{p-1}-n+1}\left(r \partial_{r}
\widetilde{u}^{\infty}+\frac{2 s+\beta}{p-1}
\widetilde{u}^{\infty}\right)^{2} d x d t \leq 0
$$
where we have used the lower semicontinuity with respect to the weak
convergence of $\left\{\widetilde{u^{\tau_{i}}}\right\}$ to
$\widetilde{u}^{\infty}$ in $H_{loc
}^{1}\left(\overline{\mathbb{R}_{++}^{n+1}} ; t^{1-2 s} d x d
t\right)$. So, it follows that
$$
\partial_{r} \widetilde{u}^{\infty}+\frac{2 s+\b}{p-1} \frac{\widetilde{u}^{\infty}}{r}=0 \quad \text { a.e. in }
\mathbb{R}_{++}^{n+1}.
$$
Therefore, we infer that $\widetilde{u}^{\infty}(X)=r^{-\frac{2
s+\b}{p-1}} \psi(\theta)$ for some function $\psi \in
H^{1}\left(S_{++}^{n} ; \theta_{1}^{1-2 s}\right)$, where $\theta=
\frac{X}{|X|}$, $\theta_1= \frac{t}{|X|}$ denotes the component of
$\theta$ in the $t$ direction and $S^n_{++}= \partial B^{++}_{1}$.\\
Hence, the claim follows.
\\
Since $\widetilde{u}^\infty$ vanishes on
$\partial\mathbb{R}^{n+1}_{++}$ from \eqref{limitpb} and using the
fact $\widetilde{u}^\infty$ is homogenous, we immediately conclude
 that $\widetilde{u}^\infty\equiv 0.$ \\\\
\textbf{Step 4.} By a Compactness result (see \cite[Lemma
2.3]{KimLee}) with \textbf{Steps 2} and \textbf{3},
$(\widetilde{u^{\tau_i}})$ converges strongly to
$\widetilde{u}^\infty\equiv 0$ in $
L^2_{loc}(\R^{n+1}_{++};t^{1-2s} dxdt)$.\\\\
 \textbf{Step 5.}
$(\widetilde{u^{\tau_i}})$ converges strongly to zero in
$H^1(B_R^{n+1} \setminus B_\epsilon ^{n+1}; t^{1-2s}dxdt)$ and
$(|x|^\frac{\b}{p+1}\widetilde{u^{\tau_i}} )$ converges strongly to
zero in $L^{p+1}((B_R^{n+1} \setminus
B_\epsilon^{n+1})\cap(\R^{n}_+\times\{0\}))$ for all $R > \epsilon >
0.$ These convergence results
follow from \eqref{2.6} and \eqref{2.7} combined with \textbf{Step 4.}  \\\\
\textbf{Step 6.} $\widetilde{u}\equiv 0$. Indeed,
\begin{eqnarray*}
E_1(\widetilde{u};\tau)&=&E_1(\widetilde{u^\tau};1)\nonumber\\
&=&\frac{1}{2}\int_{\R^{n+1}_{++}\cap B_1^{n+1}}t^{1-2s}|\nabla
\widetilde{u^\tau}|^2 dxdt+
\frac{\kappa_s \lambda \tau^{(2s+\a)}}{2}\int_{\R^{n}_+\times\{0\}\cap B_1^{n+1}}|x|^{\a}|\widetilde{u^\tau}|^2 dx\nonumber\\
&-&\; \frac{\kappa_s}{p+1}\int_{\R^{n}_+\times\{0\}\cap B_1^{n+1}}|x|^{\b}|\widetilde{u^\tau}|^{p+1}dx\nonumber\\
&=&\frac{1}{2}\int_{\R^{n+1}_{++}\cap
B_\epsilon^{n+1}}t^{1-2s}|\nabla
\widetilde{u^\tau}|^2dxdt+\frac{\kappa_s \lambda \tau^{(2s+\a)}}{2}
\int_{\R^{n}_+\times\{0\}\cap
B_\epsilon^{n+1}}|x|^{\a}|\widetilde{u^\tau}|^2
dx\nonumber\\&-&\frac{\kappa_s}{p+1} \int_{\R^{n}_+\times\{0\}\cap
B_\epsilon^{n+1}}|x|^{\b}|\widetilde{u^\tau}|^{p+1}dx +\;
\frac{1}{2}\int_{\R^{n+1}_{++} \cap B_1 ^{n+1}\setminus
B_\epsilon^{n+1}}t^{1-2s}|\nabla \widetilde{u^\tau}|^2
dxdt\nonumber\\&+&\frac{\kappa_s \lambda \tau^{(2s+\a)}}{2}
\int_{\R^{n}_+\times\{0\}\cap B_1^{n+1} \setminus
B_\epsilon^{n+1}}|x|^{\a}|\widetilde{u^\tau}|^2 dx-\;
\frac{\kappa_s}{p+1}\int_{\R^{n}_+\times\{0\}\cap B_1^{n+1} \setminus B_\epsilon^{n+1}}|x|^{\b}|\widetilde{u^\tau}|^{p+1}dx\nonumber\\
&=&\epsilon^{n-\frac{2s(p+1)+2\b}{p-1}}E_1(\widetilde{u};\epsilon\tau)+\frac{1}{2}\int_{\R^{n+1}_{++}\cap
B_1^{n+1} \setminus B_\epsilon^{n+1}}t^{1-2s}|\nabla
\widetilde{u^\tau}|^2 dxdt\nonumber\\&+&\frac{\kappa_s \lambda
\tau^{(2s+\a)}}{2}\int_{\R^{n}_+\times\{0\}\cap B_1^{n+1} \setminus
B_\epsilon^{n+1}}|x|^{\a}|\widetilde{u^\tau}|^2 dx-\;
\frac{\kappa_s}{p+1}\int_{\R^{n}_+\times\{0\}\cap B_1^{n+1}
\setminus
B_\epsilon^{n+1}}|x|^{\b}|\widetilde{u^\tau}|^{p+1}dx\nonumber\\&\leq
&
C\epsilon^{n-\frac{2s(p+1)+2\b}{p-1}}+\frac{1}{2}\int_{\R^{n+1}_{++}\cap
B_1^{n+1} \setminus B_\epsilon^{n+1}}t^{1-2s} |\nabla
\widetilde{u^\tau}|^2 dxdt\nonumber\\&+&\frac{\kappa_s \lambda
\tau^{(2s+\a)}}{2}\int_{\R^{n}_+\times\{0\}\cap B_1 ^{n+1}\setminus
B_\epsilon^{n+1}}|x|^{\a}|\widetilde{u^\tau}|^2
dx-\frac{\kappa_s}{p+1} \int_{\R^{n}_+\times\{0\}\cap B_1^{n+1}
\setminus
B_\epsilon^{n+1}}|x|^{\b}|\widetilde{u^\tau}|^{p+1}dx\nonumber.%\\&\leq
%&C\epsilon^{n-\frac{2s(p+1)+2\b}{p-1}}+C\int_{\R^{n+1}_{++}\cap
%B_1^{n+1} \setminus B_\epsilon^{n+1}}t^{1-2s} |\widetilde{u^\tau}|^2 dxdt.
\end{eqnarray*}
Letting $\tau\longrightarrow +\infty$ and then
$\epsilon\longrightarrow 0$, we deduce that
$\displaystyle{\lim_{\tau\longrightarrow +\infty}}
E_1(\widetilde{u}; \tau) = 0.$ Using the monotonicity of $E,$
$$E(\widetilde{u}; \tau)\leq \frac{1}{\tau}\int_{\tau}^{2\tau}E(t)dt\leq \sup_{[\tau,2\tau]}E_1+\tau^{-n-1+\frac{2s(p+1)+2\b}{p-1}}\int_{ B_{2\tau}^{++} \setminus B_{\tau}^{++}}(\widetilde{u})^2$$
and so \begin{equation}\label{lim0}\displaystyle
\lim_{\tau\longrightarrow +\infty} E(\widetilde{u}; \tau) = 0.
\end{equation}
Furthermore, multiply equation \eqref{0.3} by $\widetilde{u}$ and
integrate by parts to find
\begin{small}
\begin{eqnarray}\label{c}
\int_{B_1^{++}}t^{1-2s}|\nabla \widetilde{u}^\tau|^2
dxdt=\int_{\partial B_1^{++}}t^{1-2s}\partial_\nu\widetilde{u}^\tau
\widetilde{u}^\tau  d\sigma -\kappa_s\lambda \int_{
B_{1}^{+0}}|x|^{\a}|\widetilde{u}^\tau|^{2}dx + \kappa_s \int_{
B_{1}^{+0}}|x|^{\b}|\widetilde{u}^\tau|^{p+1}dx\,
\end{eqnarray}\end{small}
From \eqref{fs1} and \eqref{c}, we obtain the equivalent formula for
$E$ given by
\begin{small}
\begin{eqnarray}\label{c1}
&&E(\widetilde{u};
\tau)=\frac{p-1}{2(p+1)}\tau^{\frac{2s(p+1)+2\b}{p-1}-n}\int_{
B_{\tau}^{+0}}|x|^{\b}|\widetilde{u}|^{p+1}dx+
\tau^{\frac{2s(p+1)+2\b}{p-1}-n-1}\frac{2s+\b}{2(p-1)}\int_{\partial
B^{++}_\tau}t^{1-2s}(\widetilde{u})^2  d\sigma
\nonumber\\&&\hspace{3cm}+\frac{\tau^{\frac{2s(p+1)+2\b}{p-1}-n}}{2}\int_{\partial
B_{\tau}^{++}}t^{1-2s}\partial_\nu\widetilde{u}\widetilde{u}
d\sigma\,
\end{eqnarray}
\end{small}
By H\"older's inequality and weighted trace estimates, since %$\frac{4s+2\b}{p-1}>0$  for  $\b>-2s,\;
$ p>\frac{n+2s+2\b}{n-2s},$ we have
$$\displaystyle\lim_{\tau\to 0} \frac{\tau^{\frac{2s(p+1)+2\b}{p-1}-n}}{2} \int_{\partial B_\tau^{++}} t^{1-2s} \, \widetilde u \,
\partial_\nu \widetilde u \, d\sigma = 0.$$
Ultimately we get $\displaystyle\lim_{\tau\to 0}E(\widetilde{u};
\tau)\geq0$, and this yields $E(\widetilde{u}; \tau)\equiv 0$ for
any $\tau >0$, since $E$ is nondecreasing and taking account of
\eqref{lim0}. So, $\frac{dE}{d\tau}=0$, which means that
$\widetilde{u}$ is homogeneous. Thanks to the boundary condition, we
readily deduce that $\widetilde{u}\equiv 0$ and consequently
$u\equiv0$.
\end{pfn}\\

%\newpage
%%%%%%%%%%%%%%%%%%%%%%%%%%%%%%%%%%%%%%%%%%%%%%%%%%%%%%%%%%%%%%%%%%%%%%%%%%%%%%%%%%
%%%%%%%%%%%%%%%%%%%%%%%%%%%%%%%%%%%%%%%%%%%%%%%%%%%%%%%%%%%%%%%%%%%%%%%%%%%%%%%%%%%%%%%%%%%%%%%%%%%%%%%%%%%%%%%%%%%%%%%%%%%
%\section{The local case: $s=1$}

\section{\textbf{A monotonicity formula and some estimates in the local case}}
In the rest of the paper, we assume that $s=1$.\\
For $\tau>0$, define the function $u^\tau$ by
$$u^\tau(x)=\tau^\frac{2+\b}{p-1}u(\tau x)\;\;\mbox{for}\;\; x\in \mathcal{H},$$
then we have
\begin{prop}\label{po}
 Let $u$  be a weak solution of problem \eqref{n}. Set
 for $\tau>0$
\begin{eqnarray}\label{f1}
E(u,\tau)=\int_{\mathcal{H}\cap B_{1}^n}\Big(\frac{1}{2}|\nabla
u^\tau|^{2}+\frac{\lambda\tau^{2+\a}}{2}|x|^\a|u^\tau|^{2}-\frac{1}{p+1}|x|^\b|u^\tau|^{p+1}\Big)+
\frac{2+\b}{2(p-1)}\int_{\partial(\mathcal{H}\cap
B_{1}^n)}|u^\tau|^{2}.
\end{eqnarray}
Then
\begin{eqnarray}\label{f2}
\frac{dE}{d\tau}= \tau \int_{\partial (\mathcal{H}\cap
B_{1}^n)}(\frac{d
u^\tau}{d\tau})^2ds+\frac{(2+\a)\lambda\tau^{1+\a}}{2}
\int_{\mathcal{H}\cap B_{1}^n}|x|^\a|u^\tau|^{2}dx\;\;\mbox{for
all}\; p>1.
\end{eqnarray}
Here $B^{n}_1 $ denotes the Euclidean ball in $\R^{n}$ centered at
$0$ of radius $1$.\\ Furthermore, $E$ is a nondecreasing function of
$\tau$.
\end{prop}
\begin{pf}The proof follows the main lines of the demonstration of \cite[Proposition
 1.2]{SelmiHarrabiZaidi}, with small modifications.
\end{pf}

 In what follows, all integrals are taken over $n$-dimensional domains. For simplicity, we write  $B_R$
 instead $B_R^n$ to denote the open ball in $\R^n$ of radius $R$ centered at the origin.\\For $R>0$, denote $A_R=\{ R<|x|<2 R\}$,
 $B_R^+=B_R\cap \R^n_+$
and $\partial B_R^+=(\partial B_R)\cap\R^n_+$.\\ The key step of the
proof of Theorem \ref{th1.2} is the following integral estimates
which are useful in the subcritical, critical and supercritical
cases.
\begin{prop}\label{pp}
Let $u$  be a weak solution of problem \eqref{n} which is stable
outside a compact set $\mathcal{K} \subset \mathcal{H}\cap B_{R_0} $
for some $R_0>0$. Then for all $R> 2R_0,$ we have
\begin{eqnarray}\label{int1}
\int_{\mathcal{H}\cap B_R } \Big(|\nabla u|^2+ \l|x|^\a u^2+
|x|^\b|u|^{p+1}\Big)\leq C_0+ CR^{-2 }\int_{ \mathcal{H}\cap A_R}
u^{2},
\end{eqnarray}
 and
 \begin{eqnarray}\label{int11}
\int_{\mathcal{H}\cap B_R } \Big(|\nabla u|^2+ \l|x|^\a u^2+
|x|^\b|u|^{p+1}\Big) \leq C_0+ CR^{n-2\frac{p+\b+1}{p-1}}.
\end{eqnarray}
Here $C=C(n,p)$ and $C_0=C_0(u,\b,R_0,n,p)$ are positive constants
independent of $R$. Furthermore, if $u$ is a stable solution, then
\eqref{int1} and \eqref{int11} hold with $C_0=0$.
\end{prop}
\begin{pf} The proof of Proposition \ref{pp} is similar
to that of \cite [Proposition 1.1]{SelmiHarrabiZaidi} and we omit
the details.
\end{pf}

\begin{rmk}\label{bou-sta}
Following the same idea (see  \cite[Remark 1.3]{SelmiHarrabiZaidi}),
we can see that classification of bounded stable solutions of
problem \eqref{n} follows immediately from \eqref{int1} for all
$p>1$. We note that the study of unbounded solutions requires the
use of a monotonicity formula.
 \end{rmk}

\section{\textbf{Proof of Theorem \ref{th1.2}}}
\subsection{\textbf{Stable solutions}}
 Let $u$ be a weak stable solution of \eqref{n}.\\
If $n=2$ from \eqref{int11} we have
$$\int_{B_R^+}|x|^\b|u|^{p+1}\leq CR^{-\frac{4+2\b}{p-1}},$$
which yields that $u\equiv0$ for all $p>1$.\\
If $n\geq3$ three cases may occur.\\
{\bf Case $1$. $1<p<\frac{n+2+2\b}{n-2}$}.

From the main integral estimate \eqref{int11} of Proposition
\ref{pp} (with $C_0=0$), we have
\begin{equation*}
\int_{B_R^+}|x|^\b|u|^{p+1}\leq CR^{n-2\frac{p+\b+1}{p-1}}, \;
\forall R>0.
\end{equation*}
As $n-2\frac{p+\b+1}{p-1}<0$ if $1<p<\frac{n+2\b+2}{n-2}$, we may readily see that $u\equiv 0$.\\
{\bf Case $2$. $p=\frac{n+2+2\b}{n-2}$}.

Apply again \eqref{int11}, we derive
\begin{equation}\label{HF1}
\int_{\R_+^n}|x|^\b|u|^{\frac{ 2(n+\b)}{n-2}} < \infty \mbox{ and }
\int_{A_R\cap \R_+^n}|x|^\b|u|^{\frac{2(n+\b)}{n-2}}\to 0 \mbox{ as
} R \to \infty,
\end{equation}
where $A_R=\{ R<|x|< 2R \}$. Invoking now \eqref{int1} of
Proposition \ref{pp} (with $C_0=0$) and applying H\"{o}lder's
inequality, we derive that
\begin{align}\label{HF2}
\int_{B_R^+}|x|^\b|u|^{\frac{ 2(n+\b)}{n-2}}&\leq CR^{-2}\int_{A_R\cap \R_+^n }|u|^{2}\nonumber\\
 &\leq CR^{-2}\left(\int_{A_R}|x|^{\frac{2-n}{2+\b}\b}\right)^\frac{2+\b}{n+\b}\left(\int_{A_R\cap \R_+^n}|x|^\b
 |u|^{\frac{ 2(n+\b)}{n-2}}\right)^\frac{n-2}{n+\b}\nonumber\\
 &\leq C\left(\int_{A_R\cap \R_+^n}|x|^\b|u|^{\frac{ 2(n+\b)}{n-2}}\right)^\frac{ n-2}{n+\b}, \; \forall R>0.
\end{align}
 Using \eqref{HF1} and \eqref{HF2}, we deduce that $\displaystyle\int_{\R^n_+}|x|^\b|u|^{\frac{ 2(n+\b)}{n-2}}=0$ and so $u\equiv 0$.\\
{\bf Case $3$. $p>\frac{n+2+2\b}{n-2}$}.

This case needs more involving analysis using the powerful
monotonicity formula. The fact that $\lambda >0$ allows-us to
provide nonexistence result for all supercritical exponent $p$.
Substitute $R$ by $\tau R$ in \eqref{int11} (with $C_0=0$), then for
all  $ R>0 \mbox{ and }\tau >1$, one has
\begin{equation}\label{sca}
\int_{B_{\tau R}^+}\left(|\nabla
u|^2+|x|^\b|u|^{p+1}+\l|x|^\a|u|^{2}\right)\leq
C\tau^{n-2\frac{p+\b+1}{p-1}}R^{n-2\frac{p+\b+1}{p-1}} .
\end{equation}
Scaling back\footnote{ In fact, multiply \eqref{sca} by
$\tau^{2\frac{p+\b+1}{p-1}}$ and use the change of variable $x=\tau
y$.}, we obtain
\begin{equation}\label{rr}
\int_{B_R^+}(|\nabla
u^\tau|^2+|x|^\b|u^\tau|^{p+1}+\l\tau^{\a+2}|x|^\a|u^\tau|^{2})\leq
CR^{n-2\frac{p+\b+1}{p-1}}, \; \forall R>0, \forall \tau >1.
\end{equation}
We claim that $(u^\tau)_{\tau \geq 1}$ is uniformly bounded in
$H^1(B_R^+)$ for any $R>0$. %%%%%%%%%%%%%%%%%%%%%%%%%%%%%%%%%%%%%%%%%%%%%%%%%%%%%%%%%%%%
In fact, we distinguish three cases according to the value of the
parameter $\a$. Assume that $R>1$. If $\a=0$, the claim follows from
\eqref{rr}.\\
If $-2<\a<0$, for $x\in B_1^+$, we have $|x|^\a>1$ and for $x\in
B_R^+\setminus B_1^+$, we have $\tau^{\a+2}|x|^\a> \tau^{\a+2}R^\a
\geq 1$ for $\tau$ large enough. Thus \begin{align}\label{rr1}
\int_{B_R^+}|u^\tau|^{2}&=\int_{B_1^+}\ldots+\int_{B_R^+\setminus
B_1^+}\ldots\nonumber\\
&\leq\int_{B_1^+}|x|^\a|u^\tau|^{2}+\int_{B_R^+\setminus
B_1^+}\tau^{\a+2}|x|^\a|u^\tau|^{2}\nonumber\\
&\leq C(R),
\end{align} where we have used \eqref{rr} and $C(R)$ is a positive constant depending only on $R$. From \eqref{rr} and \eqref{rr1},
we get the desired result.\\
Lastly, if $\a>0$, we split the integral as follow
\begin{equation}\label{rr2}\int_{B_R^+}|u^\tau|^{2}=\int_{B_{\tau^{-\frac{\a+2}{\a}}}^+}\ldots+\int_{B_R^+\setminus
B_{\tau^{-\frac{\a+2}{\a}}}^+}\ldots.\end{equation} Note that $1<
\tau^{\a+2}|x|^\a$ in $B_R^+\setminus
B_{\tau^{-\frac{\a+2}{\a}}}^+$, therefore we get
\begin{equation}\label{rr3}\int_{B_R^+\setminus
B_{\tau^{-\frac{\a+2}{\a}}}^+}|u^\tau|^{2}\leq\int_{B_R^+}
\tau^{\a+2}|x|^\a|u^\tau|^{2}\leq C(R)\end{equation}by using
\eqref{rr}. Now, applying the change of variable $y=\tau x$, and
using the fact that $u\in H^1_{loc}(\R^n_+)$ and
$2\frac{2+\b}{p-1}-n<0$ we obtain \begin{equation}\label{rr4}
\int_{B^+_{\tau^{-\frac{\a+2}{\a}}}}|u^\tau(x)|^{2}dx\leq
\int_{B^+_{\tau^{-1}}}\tau^{2\frac{2+\b}{p-1}}|u(\tau
x)|^{2}dx=\tau^{2\frac{2+\b}{p-1}-n}\int_{B^+_1}|u(y)|^{2}dy \leq
C\end{equation} Combining \eqref{rr2}-\eqref{rr4} and using
\eqref{rr}, the claim follows in this case. Note that, when $R\leq
1$, a similar argument applies.\\

%%%%%%%%%%%%%%%%%%%%%%%%%%%%%%%%%%%%%%%%%%%%%%%%%%%%%%%%%%%%%%%%%%%%%%%%%%%%%%%%%%%%%%%%%%%%%%%%%%%%%%%%%%%%%
Consequently, from a standard diagonal subsequence argument, we can
find a sequence $(u^{\tau_j})$ which converges strongly in
$L^{2}(B_R^+)$ and weakly in $H^1(B_R^+)$ to some function
$u^\infty$ for every $R>0$, as $\tau_j\rightarrow +\infty$. Multiply
the following equation, satisfied by $u^\tau$,
\begin{equation*}
  u^\tau=\lambda^{-1}\tau^{-(\a+2)}|x|^{-\a}(\D u^\tau + |x|^\b|u^\tau|^{p-1}u^\tau) \; \text{ in} \; \R^{n}_{+},
\end{equation*}
by $\phi \in C_c^1(\overline{B^+_R})$, then integrate by parts and
using H\"{o}lder's inequality, we derive from \eqref{rr} that
$$|\int_{B_R^+}u^\tau \phi|\leq \dfrac{C_{R,\phi}}{\tau^{2}}, \mbox{ where } C_{R,\phi}> 0. $$
Consequently, $(u^{\tau_j})$  converges weakly to $0$ in
$L^{2}(B_R^+)$, so $u^\infty=0$ and then $(u^{\tau_j})$ converges
strongly to $0$ in $L^{2}(B_R^+)$. Invoking now inequality
\eqref{int1} of Proposition \ref{pp} where we substitute $R$ by
$\tau R$ we obtain
\begin{equation*}
\int_{B_{\tau R}^+}(|\nabla u|^2+|x|^\b|u|^{p+1}+\l|x|^\a u^{2})\leq
CR^{-2}\tau^{-2}\int_{B_{2\tau R}^+}u^2.
\end{equation*}
Thus, scaling back, we deduce
\begin{equation}\label{lp}
\int_{B_{ R}^+}(|\nabla
u^\tau|^2+|x|^\b|u^\tau|^{p+1}+\l\tau^{\a+2}|x|^\a|u^\tau|^{2})\leq
CR^{-2}\int_{B_{2 R}^+}(u^\tau)^2.
\end{equation}
As a consequence, $(u^{\tau})$ converges strongly to $0$ in
$H^1(B_R^+)$, $(\tau^\frac{\a+2}{2}|x|^\frac{\a}{2}u^\tau)$
converges strongly to $0$ in $L^{2}(B_R^+)$ and
$(|x|^{\frac{\b}{p+1}}u^\tau)$ converges to $0$ in $L^{p+1}(B_R^+)$
as $\tau\to \infty$. Then from \eqref{lp}, we have
\begin{equation}\label{lb}
\lim_{\tau\rightarrow +\infty}E_2(u,\tau)=0,
\end{equation}
where $$E_2(u,\tau):=\int_{B_{1}^+}(\frac{1}{2}|\nabla
u^\tau|^{2}+\frac{\lambda\tau^{\a+2}}{2}|x|^\a|u^\tau|^{2}-\frac{1}{p+1}|x|^\b|u^\tau|^{p+1}).$$
We claim that the same holds true for $E$. To see this, simply
observe that since $E$ is nondecreasing,
\begin{eqnarray*}
E(u,\tau)&\leq
&\frac{1}{\tau}\int_{\tau}^{2\tau}E(u,t)dt\nonumber\\&=&\frac{1}{\tau}\int_{\tau}^{2\tau}E_2(u,t)dt+\tau^{-1}\frac{2+\b}{2(p-1)}
\int_{\tau}^{2\tau}\int_{\partial
B_{1}^+}|u^t|^2\nonumber\\&\leq&\sup_{t\geq\tau}E_2(u,t)+C\int_{B_{2}^+}
|u^\tau|^2.
\end{eqnarray*}
Thanks to this, we deduce from \eqref{lb} that
$$\lim_{\tau\rightarrow +\infty}E(u,\tau)=0.$$
Furthermore, applying the same argument as in the nonlocal case, we obtain $\displaystyle\lim_{\tau\rightarrow 0}E(u,\tau)=0$.\\
Consequently, by the monotonicity of $E(u,\tau)$, we have
$E(u,\tau)\equiv 0$, and therefore $\frac{dE}{d\tau}=0$. Then from
\eqref{f2}
$$ \int_{B_{1}^+}|x|^\a|u^\tau|^{2}dx=0,\, \forall \tau >1.$$
We readily deduce that $u\equiv 0$. This completes the proof of the
first result of Theorem \ref{th1.2} in the third case.
\subsection{\textbf{Solutions that are stable outside a compact set}}

We begin by the following well known Pohozaev identity .
\begin{lem}\label{l4}
Let $u$ be a weak solution of \eqref{n}. Then the following identity
holds:
\begin{align}\label{j}
\frac{n-2}{2}&\int_{B_{R}^+}|\nabla
u|^{2}+\frac{n+\a}{2}\l\int_{B_{R}^+}|x|^\a
u^2-\frac{n+\b}{p+1}\int_{B_{R}^+}|x|^\b|u|^{p+1}
&\nonumber\\&=\frac{R}2\int_{\partial B_{R}^+}|\nabla
u|^2ds-R\int_{\partial B_{R}^+} \Big|\frac{ \partial
u}{\partial\nu}\Big|^{2}ds +\frac{R}2\l\int_{\partial B_{R}^+}|x|^\a
u^2 ds-\frac{R}{p+1}\int_{\partial B_{R}^+}|x|^\b| u|^{p+1}ds.
\nonumber\\&\;\;\;
\end{align}
\end{lem}
\begin{pf} The proof of this lemma is standard, for the completeness we
give the proof here. Multiplying the equation ($P)$ by $\langle
x,\nabla u\rangle$ and integrating on $B_R^+$, then a direct
computation shows that
\begin{align}\label{p1}
\int_{B_{R}^+}|x|^\b|u|^{p-1}u\langle x,\nabla u\rangle dx&=\frac{1}{p+1}\int_{B_{R}^+}\langle |x|^\b x,\nabla |u|^{p+1}\rangle dx\nonumber\\
&=-\frac{n+\b}{p+1}\int_{B_{R}^+}|x|^\b|u|^{p+1}dx+\frac{1}{p+1}\int_{\partial B_{R}^+}\langle x,\nu\rangle|x|^\b|u|^{p+1}ds\nonumber\\
&=-\frac{n+\b}{p+1}\int_{B_{R}^+}|x|^\b|u|^{p+1}dx+\frac{R}{p+1}\int_{\partial
B_{R}^+}|x|^\b| u|^{p+1}ds.
\end{align}
Similarly, we have
\begin{align}\label{p3}
\int_{B_{R}^+}|x|^\a u\langle x,\nabla u\rangle dx
=-\frac{n+\a}2\int_{B_{R}^+}|x|^\a u^2dx+\frac{R}2\int_{\partial
B_{R}^+}|x|^\a u^2 ds.
\end{align}
Since $u \in C^2(\overline{\R_+^n})$ and $u=0$ on $\partial\R^n_+$,
then $\nabla u.x=0$ on $\partial\R^n_+$. Hence, integrate by parts
yields
\begin{align}\label{p2}
\int_{B_{R}^+}-\D u\langle x,\nabla u\rangle dx&=\int_{B_{R}^+}\nabla u\nabla (\langle x,\nabla u\rangle) dx-\int_{\partial B_{R}^+}\frac{\partial u} {\partial\nu}\langle x,\nabla u\rangle ds\nonumber\\
&=\int_{B_{R}^+}\nabla u\nabla (\langle x,\nabla u\rangle) dx-R\int_{\partial B_{R}^+} \Big|\frac{ \partial u}{\partial\nu}\Big|^{2}ds\nonumber\\
&=\int_{B_{R}^+}|\nabla u|^2dx+\frac{1}{2}\int_{B_{R}^+}\langle x,\nabla(|\nabla u|^2)\rangle dx-R\int_{\partial B_{R}^+} \Big|\frac{ \partial u}{\partial\nu}\Big|^{2}ds\nonumber\\
&=\int_{B_{R}^+}|\nabla u|^2dx-\frac{n}{2}\int_{B_{R}^+}|\nabla u|^2dx+\frac{R}2\int_{\partial B_{R}^+}|\nabla u|^2ds-R\int_{\partial B_{R}^+} \Big|\frac{ \partial u}{\partial\nu}\Big|^{2}ds\nonumber\\
&=-\frac{n-2}{2}\int_{B_{R}^+}|\nabla u|^2dx+\frac{R}2\int_{\partial B_{R}^+}|\nabla u|^2ds-R\int_{\partial B_{R}^+} \Big|\frac{ \partial u}{\partial\nu}\Big|^{2}ds.\nonumber\\
\end{align}
Then \eqref{j} follows immediately from \eqref{p1}, \eqref{p3} and
\eqref{p2}.
\end{pf}

We turn now to the proof of the second result of Theorem
\ref{th1.2}. Assume that $n\geq 3$. Arguing by contradiction, we
suppose that $u\neq0$ then, by virtue of the main integral estimate
\eqref{int11} of Proposition \ref{pp} (with $C_0> 0$) we have
$$ \int_{\R^n_+}|x|^\b|u|^{p+1}, \; \int_{\R^n_+}|x|^\a u^2 \mbox{ and } \int_{\R^n_+}|\nabla u|^2<\infty; \mbox{ for any } p\in \Big(1,\frac{n+2\b+2}{n-2}\Big).$$
Consequently, we can find a sequence $R_k \to \infty$  such that
\begin{align*}
&\frac{R_k}2\int_{\partial B_{R_k}^+}|\nabla
u|^2ds-R_k\int_{\partial B_{R_k}^+} \Big|\frac{ \partial
u}{\partial\nu}\Big|^{2}ds +\frac{R_k}2\l\int_{\partial
B_{R_k}^+}|x|^\a u^2
ds\nonumber\\&\;\;\;-\frac{R_k}{p+1}\int_{\partial B_{R_k}^+}|x|^\b|
u|^{p+1}ds
\end{align*}
tends to $0$ as $k$ tends to $+\infty$.
 So we deduce from the Pohozaev identity \eqref {j} that
\begin{equation*}
\frac{n-2}{2}\int_{\R^n_+}|\nabla
u|^{2}+\frac{n+\a}{2}\l\int_{\R^n_+}|x|^\a
u^2=\frac{n+\b}{p+1}\int_{\R^n_+}|x|^\b|u|^{p+1}.
\end{equation*}
Since, we have assumed that
$\big||x|^\frac{\b-\a}{p-1}u\big|_{L^\infty(\mathcal{H})}^{p-1}<\frac{\lambda
(n+\a)(p+1)}{2(n+\b)}$, then the last equality implies
$$\l\int_{\R^n_+}|x|^\a u^2\leq \frac{2(n+\b)}{(n+\a)(p+1)}\int_{\R^n_+}|x|^\b|u|^{p+1} <\lambda\int_{\R^n_+}|x|^\a u^2.$$
 So, we reach a contradiction which completes the proof $2.(a)$ of Theorem \ref{th1.2}.

Assume that $n\geq3$ and $p\geq \frac{n+2\b+2}{n-2}$. Let $u$ be a weak solution of \eqref{n} which is stable outside a compact set. We discuss two cases.\\
\textbf{{Case 1}.} $p=\frac{n+2\b+2}{n-2}$.

In view of the Pohozaev identity and as in the proof of Theorem
\ref{th1.2}, we have
$$\int_{\R^n_+}|x|^\b|u|^{\frac{2(n+\b)}{n-2}}, \; \int_{\R^n_+}|x|^\a u^2 \mbox{ and } \int_{\R^n_+}|\nabla u|^2<\infty,$$
and
\begin{equation}\label{p0}
\int_{\R^n_+}|\nabla u|^{2}+\frac{n+\a}{n-2}\l\int_{\R^n_+}|x|^\a
u^2=\int_{\R^n_+}|x|^\b|u|^{\frac{2(n+\b)}{n-2}}.
\end{equation}
 Now, multiplying equation \eqref{n} by $u$ and integrating by parts over $B_{R}^+$ yields
\begin{equation}\label{sq}
\int_{B^+_{R}}|\nabla u|^{2} +\lambda\int_{B^+_{R}}|x|^\a
u^2-\int_{B^+_{R}}|x|^\b|u|^{\frac{2(n+\b)}{n-2}}
 =\int_{\partial B^+_{R}} u \nabla u.\frac{x}{|x|}\leq \int_{\partial B^+_{R}} (u^2+ |\nabla u|^2).
\end{equation}
Since we have $$ \int_{\R^n_+}|x|^\a u^2 \mbox{ and }
\int_{\R^n_+}|\nabla u|^2<\infty,$$ then we can find a sequence $R_k
\to \infty$  such that
 $$\int_{\partial B^+_{R_k}} u^2+ |\nabla u|^2 \to 0 \,\mbox{  as } \,k \to \infty.$$
Hence, from \eqref{sq}, we deduce that
\begin{equation}\label{nl}
\int_{\R^n_+}|\nabla u|^{2}+\l\int_{\R^n_+}|x|^\a
u^2=\int_{\R^n_+}|x|^\b|u|^{\frac{2(n+\b)}{n-2}}.
\end{equation}
Combining \eqref{p0} and \eqref{nl} and using the fact that $\a>-2$,
we derive that
$$\l\int_{\R^n_+}|x|^\a u^2=0.$$
Since $\l>0$, we must have $u\equiv 0$.\\
\textbf{{Case 2}.} $p>\frac{n+2\b+2}{n-2}$.

Substitute $R$ by $\tau R$ in \eqref{int11} (with $C_0> 0$) and
scaling back, we obtain
\begin{equation*}
\int_{B_R^+}(|\nabla
u^\tau|^2+|x|^\b|u^\tau|^{p+1}+\l\tau^{2+\a}|x|^\a|u^\tau|^{2})\leq
C_0 \tau^{2\frac{p+\b+1}{p-1}-n}+ CR^{n-2\frac{p+\b+1}{p-1}}, \;
\forall R>0, \forall \tau >1.
\end{equation*}
Since $2\frac{p+\b+1}{p-1}-n<0$, by arguing as in the case of stable
solutions, we derive that $(u^\tau)_{\tau \geq 1}$ is uniformly
bounded in $H^1(B_R^+)$ for any $R>0$. So, we can find a sequence
$(u^{\tau_j})$ which converges strongly in $L^{2}(B_R^+)$ and weakly
in $H^1(B_R^+)$ to some function $u^\infty$ for every $R>0$, as
$\tau_j\rightarrow +\infty$. Since $u^{\tau_j}$ satisfies the
following PDE

\begin{eqnarray*}
  u^\tau=\lambda^{-1}\tau^{-(2+\a)}|x|^{-\a}(\D u^\tau + |x|^\b|u^\tau|^{p-1}u^\tau) \; &\text{ in \; $\R^{n}_{+}$},
\end{eqnarray*}
then taking limits in the sense of distributions, we get
$u^\infty=0$, and therefore $(u^{\tau_j})$ converges strongly to $0$
in $L^{2}(B_R^+)$. Invoking now inequality \eqref{int1} of
Proposition \ref{pp} where we substitute $R$ by $\tau R$ we obtain
\begin{equation*}
\int_{B_{\tau R}^+}(|\nabla u|^2+|x|^\b|u|^{p+1}+\l|x|^\a u^{2})\leq
C_0+ CR^{-2}\tau^{-2}\int_{B_{2\tau R}^+}u^2,\; \forall R>2R_0
\mbox{ and }\tau >1.
\end{equation*}
Thus, scaling back, we deduce
\begin{equation*}
\int_{B_{ R}^+}(|\nabla
u^\tau|^2+|x|^\b|u^\tau|^{p+1}+\l\tau^{2+\a}|x|^\a|u^\tau|^{2})\leq
C_0\tau^{2\frac{p+\b+1}{p-1}-n }+CR^{-2}\int_{B_{2 R}^+}(u^\tau)^2.
\end{equation*}
As a consequence, $(u^{\tau})$ converges strongly to $0$ in
$H^1(B_R^+)$,  $(\tau^\frac{2+\a}{2}|x|^\frac{\a}{2} u^\tau)$
converges strongly to $0$ in $L^{2}(B_R^+)$ and
$(|x|^{\frac{\b}{p+1}}u^{\tau})$ converges to $0$ in
$L^{p+1}(B_R^+)$ as $\tau\to \infty$. Set
$$E_2(u,\tau):=\int_{B_{1}^+}(\frac{1}{2}|\nabla u^\tau|^{2}+\frac{\lambda\tau^{2+\a}}{2}|x|^\a|u^\tau|^{2}-\frac{1}{p+1}|x|^\b|u^\tau|^{p+1}).$$
 We may easily see that
\begin{equation*}
\lim_{\tau\rightarrow +\infty}E_2(u,\tau)=0.
\end{equation*}
At this stage, the remainder of the proof is similar to that of the
first result in Theorem \ref{th1.2}, concerning stable solutions, in
the third case. We therefore omit the details.
\\\\

%%%%%%%%%%%%%%%%%%%%%%%%%%%%%%%%%%%%%%%%%%%%%%%%%%%%%%%%%%%%%%%%%%%%%%%%%%%%%%%%%%%%%%%%%%%%%%%%%%%%%%%%%%%%%%%%%%%%%%%%%%%%%%%
%\section*{Funding}
%Not applicable.\\

%\section*{Availability of data andmaterials}
%Not applicable.\\

%\section*{Competing interests}
%The authors declare that they have no competing interests regarding the publication of this paper.\\

%\section*{Authors' contributions}
%All authors contributed to the writing of this paper. All authors read and approved the final manuscript.\\
%%%%%%%%%%%%%%%%%%%%%%%%%%%%%%%%%%%%%%%%%%%%%%%%%%%%%%%%%%%%%%%%%%%%%%%%%%%%%%%%%%%%%%%%%%%%%%%%%%%%%%%%%%%%%%%%%%%%%%%%%%%%%%%%%%%%%%%%%%%%%%%%

\end{document}